\documentclass[oneside,12pt]{amsart}
\usepackage{amsmath,amsfonts, latexsym,amssymb}
\usepackage[active]{srcltx}
\newtheorem{theorem}{Theorem}[section]
\newtheorem{lemma}[theorem]{Lemma}
\newtheorem{corollary}[theorem]{Corollary}
\newtheorem{proposition}[theorem]{Proposition}
\newtheorem{conjecture}[theorem]{Conjecture}
    \theoremstyle{definition}
\newtheorem{definition}[theorem]{Definition}
\newtheorem{example}[theorem]{Examples}

\newtheorem{remark}[theorem]{{\it Remark}}

\numberwithin{equation}{section}


\begin{document}

\title[]{Generalizations  of Carmichael  numbers I}

\author{Romeo Me\v strovi\' c}

\address{Maritime Faculty, University of Montenegro, Dobrota 36,
 85330 Kotor, Montenegro} \email{romeo@ac.me}

{\renewcommand{\thefootnote}{}

\setcounter{footnote}{0}}

\dedicatory{Dedicated  to the $100$th anniversary of birth of Paul Erd\H{o}s}


\vspace*{7mm}

   \begin{abstract} 
A composite positive integer $n$ is said to be a 
{\it weak Carmichael number} if 
  $$
 \sum_{\gcd(k,n)=1\atop 1\le k\le n-1}k^{n-1}\equiv \varphi(n) \pmod{n}.
\leqno(1)
    $$
It is proved that a composite positive
integer $n$ is a weak Carmichael number 
if and only if $p-1\mid n-1$ for every prime divisor $p$ of $n$. 
This together with  Korselt's criterion yields
the fact that every Carmichael number is also a weak Carmichael number.

In this paper we mainly investigate arithmetic properties of 
weak Carmi-\
chael numbers. Motivated by the investigations of Carmichael numbers
in the last hundred years, here 
we establish several related results, notions, 
examples and computatinoal searches for weak Carmichael numbers and numbers closely 
related to weak Carmichael numbers. 
Furthermore, using the software {\tt Mathematica 8}, we present the table 
containing all non-prime powers weak Carmichael numbers less than $2\times 10^6$.

Motivated by heuristic arguments, our computations and some old  conjectures 
and results for Carmichael numbers, we propose several conjectures for weak Carmichael 
numbers and for some other classes of Carmichael like numbers. 

Finally, we consider weak Carmichael numbers in light of Fermat 
primality test. We believe  that it can be of interest 
to involve certain particular classes of weak Carmichael numbers in some 
problems concerning Fermat-like primality tests and the generalized Riemann
hypothesis.
  \end{abstract}

\maketitle

\section{Carmichael numbers, Lehmer numbers and Giuga numbers}

\subsection{Lehmer numbers, Carmichael numbers and the main result}

{\it Lehmer's totient problem} asks about the existence of a composite 
number such that $\varphi(n)\mid (n-1)$ \cite{leh}, where 
$\varphi(n)$ is the  {\it Euler totient function}
defined as the number of positive integers less than $n$ which are relatively 
prime to $n$.
 These numbers are sometimes 
reffered to  as {\it Lehmer numbers}. In 1932 D.H.  Lehmer \cite{leh}
showed that every Lehmer number $n$ must be odd and square-free, and 
that the number of distinct prime factors of $n$ must be greater than 6.
However, no Lehmer numbers are known up to date, and computations by Pinch 
\cite{pi} show that any examples must be greater than $10^{30}$.  
In 1977 Pomerance \cite{pom} showed that the number of Lehmer numbers $n\le x$
is $O(x^{1/2}(\log x)^{3/4})$. In 2011 this bound is improved by 
Luca and Pomerance \cite{lp} to $O(x^{1/2}(\log x)^{1/2+o(1)})$.

Carmichael numbers are quite famous among specialists in number theory,
as they are quite rare and very hard to test.
{\it Fermat little theorem} says that if $p$ is a prime and the integer $a$ is 
not a multiple of $p$, then $a^{p-1}\equiv 1(\bmod{\,p})$. 
However, there are positive integers $n$ that are composite but still satisfy 
the congruence  $a^{n-1}\equiv 1(\bmod{\,n})$ for all $a$ coprime to $n$.
Such ``false primes" are called {\it Carmichael numbers} in honour 
of R.D. Carmichael, who demonstrated their existence in 1912 \cite{car2}.
A Carmichael number $n$ is a composite integer that is
a base-$a$ {\it Fermat-pseudoprime} for all $a$ with $\gcd (a,n)=1$. 
These numbers present a major problem for Fermat-like primality tests.
In \cite{gr3} A. Granville wrote: 
``{\it Carmichael numbers are nuisance, masquerading as primes like this, 
though computationally they only appear rarely. Unfortunately it 
was recently proved that there are infinitely many 
of them and that when we go out far enough they are 
not so rare as it first appears.}"

It is easy to see that every Carmichael number is odd, namely, if 
$n\ge 4$ is even, then $(n-1)^{n-1}\equiv (-1)^{n-1}=-1\not\equiv -1
(\bmod{n})$. 
In 1899 A. Korselt \cite{kor} gave a complete
characterization of  Carmichael numbers which is often rely on 
the following equivalent definition.
 \begin{definition}[Korselt's criterion, 1899]
A composite odd positive integer $n$ is a Carmichael number if $n$ 
is  squarefree, and $p-1\mid n-1$ for every prime $p$ dividing $n$.
   \end{definition}
Korselt did not find any Carmichael numbers, however.
The smallest Carmi-\
chael number, $561(=3\times 11\times 17)$, was found 
by Carmichael in 1910 \cite{car1}. Carmichael also gave 
a new characterization of these numbers as those composite 
$n$ which satisfy $\lambda(n)\mid n-1$,
where $\lambda(n)$, {\it Carmichael lambda function},
denotes the size of the largest cyclic subgroup of 
the group $\left(\Bbb Z/n\Bbb Z\right)^{*}$ of all reduced residues modulo $n$. 
In other words, $\lambda(n)$ is the smallest 
positive integer $m$ such that $a^m\equiv 1(\bmod{\,n})$ for all
for all $a$ coprime to $n$ (Sloane's sequence A002322).
Since $\lambda(n)\mid \varphi(n)$ for every positive integer $n$, 
every Lehmer number would also be a Carmichael number.
Recall that various upper bound and lower bounds for $\lambda(n)$
have been obtained in \cite{eps}.
It is easily deduced from Korselt's criterion that every 
Carmichael number is a product of at least three distinct primes
(see e.g., \cite{gp}).  It was unsolved problem for many  years whether there 
are infinitely many Carmichael numbers. The question was resolved in 1994
by  Alford, Granville and Pomerance \cite{agp} who  proved, not only 
that the answer is yes, but that 
there are more than $x^{2/7}$ Carmichael numbers up to $x$, once $x$ is 
sufficiently large.
In 2005 G. Harman \cite{har} has improved the 
constant $2/7$ to $0.33$ (for a more general result see 
\cite[Theorem 1.2]{har2}). 
However, there are a very wide gap between these 
estimates and the known upper bounds for $C(x)$. Related upper bounds and the 
counting function for the Carmichael numbers were studied in 1956 by
P. Erd\H{o}s \cite{er1}, in 1980 by C. Pomerance, J.L. Selfridge and 
Samuel S. Wagstaff \cite{psw} and in 1989 by C. Pomerance \cite{pom3}.
In the same paper Erd\H{o}s proposed a popular method for the 
construction of Carmichael numbers (cf. \cite{zh} and for a 
recent application of this construction see \cite{gp} and \cite{ln}).
Some other algorihms for constructing Carmichael numbers 
can be found in \cite{aghs} and \cite{ln} where are constructed 
Carmichael numbers with millions of  components. Recall also 
that in 1939, Chernick \cite{ch} gave a simple method to obtain
Carmichael numbers with three prime factors considering the products 
of the form $(6m+1)(12m+1)(18m+1)$ with $m\ge 1$.  
Notice also that the number of Carmichael numbers less than $10^n$ is given 
in \cite{sl} as the Sloane's sequence A055553.

Quite recently, T. Wright \cite{wr}  proved that for 
every pair of coprime positive integers $a$ and $d$ there are  
infinitely many Carmichael numbers $m$ such that $m\equiv a (\bmod{\,d})$.
 \begin{remark}
Quite recently, J.M. Grau and A.M. Oller-Marc\'en 
 \cite[Definition 1]{gm2} weakened Lehmer property by introducing
the concept of $k$-Lehmer numbers.  For given  positive integer $k$, a  
$k$-{\it Lehmer number} is a composite integer $n$ such that 
$\varphi(n)\mid (n-1)^k$. It is easy to see that every $k$-{\it Lehmer number}
must be square-free. 
Hence, if we denote by $L_k$ the set that each 
$k$-{\it Lehmer number}
  $$
L_k:=\{n\in\Bbb N:\, \varphi(n)\mid (n-1)^k\},
  $$
then $k$-Lehmer numbers are the composite elements of $L_k$.
Then  $L_k\subseteq L_{k+1}$ for each $k\in \Bbb N$, and define
  $$
L_{\infty}:=\bigcup_{k=1}^{\infty}L_k.
  $$
Then it can be easily shown  that \cite[Proposition 3]{gm2}
   $$
L_{\infty}:=\{n\in\Bbb N:\,{\rm rad}(\varphi(n))\mid (n-1)\}.
  $$
This immediately  shows that \cite[Proposition 6]{gm2}
if $n$ is a Carmichael number, then $n$ also belongs to the set 
$L_{\infty}$. This leads  to the following characterization of 
Carmichael numbers which slightly modifies Korselt's criterion.
 \end{remark}
 \begin{proposition}{\rm(}\cite[Proposition 6]{gm2}{\rm)}
A composite number $n$ is a Carmichael number if and only if 
${\rm rad}(\varphi(n))\mid n-1$ and $p-1\mid n-1$ for every prime 
divisor $p$ of $n$.
 \end{proposition}

 Obviously, the composite elements of $L_1$ are precisely the Lehmer numbers
and the Lehmer property asks whether $L_1$ contains composite numbers
or not. Nevertheless, for all  $k>1$, $L_k$ always contains composite 
elements (cf. Sloane's sequence A173703 in OEIS \cite{sl} 
which presents $L_2$).  
For further radically weaking the Lehmer and Carmichael conditions 
see \cite{mcn}.
 \begin{remark} Carmichael numbers can be generalized using concepts 
of abstract algebra. Namely, in 2000 Everet W. Howe \cite{ho} defined a Carmichael number 
of order $m$ to be a composite integer $n$ such that $n$th power raising 
defines an endomorphism of every $\Bbb Z/n\Bbb Z$-algebra that can be 
generated  as a $\Bbb Z/n\Bbb Z$-module by $m$ elements. The author gave
a simple criterion to determine whether a number is a Carmichael number of
order $m$. In 2008 G.A. Steele \cite{ste} generalized Carmichael numbers
to ideals in number rings and proved a generalization of Korselt's 
criterion for these Carmichael ideals. 
  \end{remark}
Here, as always in the sequel, $\gcd(k,n)$
denotes the {\it greatest common divisor} of $k$ and $n$, and
$\sum_{\gcd(k,n)=1\atop k\in {\mathcal P}}\cdot$ 
denotes the sum ranging over all integers $k$ satisfying 
the prperty ${\mathcal P}$ and the condition $\gcd(k,n)=1$. 

Studying 
some variations on the ``theme of Giuga", in 1995 J.M. Borwein and E. 
Wong \cite{bw} established the following result.

 \begin{theorem}{\rm(\cite[Corollary 8]{bw})}
 A positive integer $n\ge 2$ satisfies the congruence
 \begin{equation}      
 \sum_{\gcd(k,n)=1\atop 1\le k\le n-1}k^{n-1}\equiv 
\varphi(n) \pmod{n}
 \end{equation}
if and only if $p-1\mid n-1$ for every prime divisor $p$ of $n$. 
 \end{theorem}
\begin{remark} Theorem 1.5 was proved in \cite{bw} as a particular case of  
Theorem 11 in \cite{bw}. In the proof of this theorem 
the authors deal with  congruences for the sum (1.1)  modulo prime 
powers dividing $n$. In particular, in this proof it was used the 
{\it Chinese remainder theorem} to factor the sum  (1.1) modulo $n$ into 
product of $s$ similar ``restricted sums", where $s$ is a number of distinct 
prime factors of $n$. In Section 4 we give another proof of Theorem 1.5 
(this  is in fact proof of Theorem 2.4). 
Our proof is  based on some congruential properties 
of sums of powers  $\sum_{\gcd(k,n)=1\atop 1\le k\le n-1}k^{n-1}$
(Lemmas 4.1--4.7) and Carlitz-von Staudt's result \cite{ca} for determining 
$S_{2k}(m)(\bmod{\,m})$ (Lemma 4.8). 
   \end{remark}

A direct consequence of Theorem 1.5 is the following 
simple characterization of Carmichael numbers. 
 \begin{corollary} {\rm(Corollary 2.8).}  
A composite positive integer $n$ is a Carmichael number if 
and only if the following conditions are satisfied.

$\,\, (i)$ $n$ is square-free and 

$(ii)$ $\displaystyle\sum_{\gcd(k,n)=1\atop 1\le k\le n-1}k^{n-1}\equiv 
\varphi(n) \pmod{n}$.
  \end{corollary}

In this paper we mainly investigate arithmetic properties of 
composite positive integers satisfying the congruence (1.1).
Such numbers are called weak Carmichael numbers.
Motivated by the investigations of Carmichael numbers
in the last hundred years, here 
we establish several related results, notions, 
examples and computatioal searches for weak Carmichael numbers and 
numbers closely related to weak Carmichael numbers. 

  \subsection{Bernoulli's  formula for the sum of powers
 and von Staudt-Clau-\
sen's theorem} 
The sum of powers of integers $\sum_{i=1}^ni^k$
 is a well-studied problem in mathematics (see e.g., \cite{bo}, \cite{sc}).
Finding formulas for these sums has interested mathematicians 
for more than 300 years since the time of James Bernoulli (1665-1705).
These lead to numerous recurrence relations.
The first such well known recurrence relation was established by 
B. Pascal \cite{pas}.
A related new reccurrence relation is 
quite recently established in \cite[Corollary 1.9]{me2}.
For a nice account of  sums of powers see \cite{ed3}.
For simplicity, here 
as often  in the sequel, for all integers $k\ge 1$ and $n\ge 2$ we denote
  $$
S_k(n):=\sum_{i=1}^{n-1}i^{k}=1^{k}+2^{k}+3^{k}+\cdots+(n-1)^{k}.
   $$ 
The study of these sums led Jakob Bernoulli \cite{ber}
to develop numbers later named in his honor. Namely, the celebrated
{\it Bernoulli's  formula} (sometimes called {\it Faulhaber's formula})
(\cite{fau} and \cite{kn}) gives the sum $S_k(n)$ explicitly as 
(see e.g., \cite{gkp} or  \cite{bea})
    \begin{equation}
S_k(n)= \frac{1}{k+1} \sum_{i=0}^{k}{k+1\choose i} n^{k+1-i}B_{i}
    \end{equation}
where $B_i$ ($i=0,1,2,\ldots$)
are {\it Bernoulli numbers} defined by the generating function
   $$
\sum_{i=0}^{\infty}B_i\frac{x^i}{i!}=\frac{x}{e^x-1}.
  $$
It is easy to find the values $B_0=1$, $B_1=-\frac{1}{2}$, 
$B_2=\frac{1}{6}$, $B_4=-\frac{1}{30}$, and $B_i=0$ for odd $i\ge 3$. 
Furthermore, $(-1)^{i-1}B_{2i}>0$ for all $i\ge 1$. 
Recall that several identities involving Bernoulli numbers and 
Bernoulli polynomials can be found in \cite{ps} and \cite{sp}.

 The {\it von Staudt-Clausen's theorem} is a result determining the fractional
part of Bernoulli numbers, found in 1840  independently by K. von Staudt
(\cite {st}; see also \cite[Theorem 118]{hw}) and T. Clausen \cite{cl}).
   \begin{theorem}{\rm (von Staudt-Clausen's theorem).}
The denominator of  Bernoulli number $B_{2n}$ with $n=1,2,\ldots$ 
is the product of all primes $p$ such that $p-1$ divides $2n$.
 \end{theorem}
 \begin{remark}
In literature, von Staudt-Clausen's theorem is often formulated as:
   $$
B_{2n}+\sum_{p-1\mid 2n\atop p{\rm\,\, prime}}\frac{1}{p}\quad
{\rm is\,\, an\,\, integer\,\, for\,\, each\,\,} n=1,2,\ldots,
   $$
or equivalently (see e.g.,  \cite[page 153]{su}):
  \begin{equation*}
pB_{2n}\equiv \left\{
  \begin{array}{ll}
\,\,\,\, 0\,(\bmod{\,p}) & {\rm if}\,\,p-1\nmid 2n\\
-1\,(\bmod{\,p}) & {\rm if}\,\,p-1\mid 2n,
\end{array}\right.
  \end{equation*}
where $p$ is a prime and $k$ a positive integer.

We also point out that in the proof of Theorem 2.4
we use a particular case of a Carlitz-von Staudt's result (see Remark 1.6)
which can be easily deduced from the above form of von Staudt-Clausen's 
theorem.
  \end{remark}
  \subsection{Giuga's conjecture and Giuga numbers}
Notice that if $n$ is any prime, then  by Fermat's little theorem,
$S_{n-1}(n)\equiv -1 (\bmod{\,n})$.  
In 1950 G. Giuga \cite{gi} proposed that the converse is also true 
via the following conjecture.
 \begin{conjecture}[Giuga's conjecture]
A positive integer
$n\ge 2$ is a prime if and only if  
 \begin{equation}
S_{n-1}(n):=\sum_{i=1}^{n-1}i^{n-1}\equiv -1 \pmod{n}.
 \end{equation}
 \end{conjecture}
A counterexample to {\it Giuga's conjecture} is called 
a {\it Giuga  number}. It is easy to show that $S_{n-1}(n)\equiv -1\, 
(\bmod{\,n})$ if and only if for each prime divisor $p$ of $n$, $(p-1)\mid (n/p-1)$
and $p\mid (n/p-1)$ (see \cite{gi}, \cite[Theorem 1]{bbb} or 
\cite[p. 22]{r2}). 
Observe  that both these conditions are equivalent to the condition that
$p^2(p-1)\mid p(n-1)$.  
Therefore, any Giuga  number must be squarefree. 
Giuga \cite{gi} showed that there are no exceptions to the conjecture
up to $10^{1000}$. In 1985 E. Bedocchi \cite{bed} improved this bound 
to $n>10^{1700}$. In 1996 D. Borwein, J.M. Borwein, 
P.B. Borwein and R. Girgensohn \cite{bbb} raised  the bound to $n>10^{13887}$. 
In 2011 F. Luca, C. Pomerance and I. Shparlinski \cite{lps}
proved that for any real number $x$, the number of counterexamples to
Giuga's conjecture $G(x):=\#\{n<x:\, n\,\,\mbox{ is composite and}
\,\,S_{n-1}(n)\equiv -1\, (\bmod{\,n})\}$  satisfies the estimate 
$G(x)=O(\sqrt{x}/(\log x)^2)$  as $x\to\infty$ improving slightly 
on a previous result by V. Tipu \cite{t}.
Quite recently,  J.M. Borwein, M. Skerritt and C. Maitland
\cite[Theorem 2.2]{bsm} reported that any counterexample to Giuga's 
primality conjecture is an odd square-free integer with at least 4771 
prime factors and so must exceed $10^{19907}$. 

Let $\varphi(n)$ be the {\it Euler totient function}. 
 \begin{definition}
A positive composite integer $n$ is said to be a {\it Giuga number} if 
  \begin{equation}
 \sum_{k=1}^{n-1}k^{\varphi(n)}\equiv -1 \pmod{n}.
  \end{equation}
\end{definition}
This definition was given by Giuga \cite{gi}.
However, it is known (e.g., see \cite[Theorem 1]{bbb}) that 
a positive composite integer $n$ is a {\it Giuga number} if 
and only if $p^2(p-1)$ divides $n-p$ 
for every prime divisor $p$ of $n$. Moreover, it is easy to see that only 
square-free integers can be Giuga numbers.
 For more information about Giuga numbers see  
D. Borwein et al. \cite{bbb}, J.M. Borwein and E. Wong \cite{bw},  
and E. Wong \cite[Chapter 2]{w}. 

A {\it weak Giuga number} is a composite number $n$ for which the sum
 $$
-\frac{1}{n}+\sum_{p\mid n\atop p\,\,{\rm prime}}\frac{1}{p}
 $$
is an integer. It is known that each Giuga number is a weak Giuga number
and  that $n$ is  a weak Giuga number
if and only if $p^2\mid n-p$ for every prime divisor $p$ of $n$
(see \cite{bbb}).
Up to date only thirteen weak Giuga numbers are known and all these numbers 
are even. The first few Giuga  numbers are
$30, 858, 1722, 66198, 2214408306, 24423128562, 432749205173838$ 
\noindent (see  sequence A007850 in \cite{sl}).

Independently, in  1990 T. Agoh (published in 1995 \cite{ag}; 
see also \cite{bw} and Sloane's sequence A046094 in 
 \cite{sl})  proposed the following conjecture.
  \begin{conjecture}{\rm(Agoh's conjecture)}.
A positive integer $n\ge 2$ is a prime if and 
only if  $nB_{n-1}\equiv -1 (\bmod{\,n})$.
     \end{conjecture}
 \begin{remark}
Notice  that the denominator of the  number $nB_{n-1}$ can be greater than 1, 
but since by  {\it von Staudt-Clausen's theorem} (Theorem 1.8), 
the denominator of any Bernoulli number $B_{2k}$ is squarefree, it follows 
that  the denominator of $nB_{n-1}$ 
is invertible modulo $n$. In 1996  it was  
reported by  T. Agoh \cite{bbb} that his conjecture is equivalent 
to Giuga's conjecture, hence the name Giuga-Agoh's conjecture found in the
litterature.  Therefore, 
   \end{remark}
 \begin{proposition}
Giuga's conjecture and Agoh's conjecture are equivalent. 
 \end{proposition}

It was pointed out in \cite{bbb} 
that this can be seen  from the Bernoulli formula (1.2) after some
analysis involving von Staudt-Clausen's theorem.
The equivalence of both conjectures is in details proved 
in 2002 by B.C. Kellner \cite[Satz 3.1.3, Section 3.1, p. 97]{kel1} 
(also see \cite[Theorem 2.3]{kel2}).  
In a recent manuscript  \cite[Subsection 2.1]{me1} 
the author of this article  proposed several Giuga-Agoh's-like conjectures. 

Notice that von Staudt-Clausen's theorem allows one 
to give the following equivalent reformulation of  Korselt's 
criterion involving the {\it Bernoulii number} $B_{n-1}$ is 
(see e.g., \cite[Section 2, Remarks after Proposition 2]{psw}, \cite{we}).
 \begin{definition}
An odd composite positive integer $n$ is a Carmichael number if and only if
$n$ is  squarefree and $n$ divides the denominator of the Bernoulli number 
$B_{n-1}$.
   \end{definition}
We present the following relationship  between Giuga's conjecture and
Carmi-\
chael numbers. 
 \begin{proposition}{\rm(see e.g.,} \cite[Theorem]{bbb} {\rm or} 
\cite[Corollary 4]{gm}{\rm )}
A positive integer $n$ is a counterexample to Giuga's conjecture
if and only if it is both a Carmichael and a Giuga number.
 In other words, a positive integer $n$ satisfies the congruence
 \begin{equation}
\sum_{i=1}^{n-1}i^{n-1}\equiv -1 \pmod{n}
 \end{equation} 
if and only if is $n$ is both a Carmichael and a Giuga number.
 \end{proposition}

In 2011 J.M. Grau and A.M. Oller-Marc\'en \cite{gm} established
a new approach to  Giuga's conjecture as follows.
 \begin{proposition}{\rm(}\cite[Corollary 3]{gm}{\rm)}
If a positive integer $n$ is a counterexample to Giuga's conjecture,
then for each positive integer $k$ 
 \begin{equation}
\sum_{i=1}^{n-1}i^{k(n-1)}\equiv -1 \pmod{n}.
 \end{equation} 
  \end{proposition}
\begin{remark} 
Proposition 1.17 leads to the generalization of Giuga's ideas in the following
way \cite[Section 3]{gm}: Do there exist integers $k$ such that the
congruence (1.6) is satisfied by some composite integer $n$?
Several open problems concerning Giuga's conjecture
can be found in J.M. Borwein and E. Wong \cite[8, E Open Problems]{bw}.
 \end{remark}
\begin{remark} 
Quite recently, J.M. Grau and A.M. Oller-Marc\'en 
\cite[Theorem 1]{gm3} characterized, in terms of the prime divisors of
$n$, the pairs $(k,n)$ for which $n$ divides $S_k(n)$. More generally,
in \cite{gm3} it is investigated $S_{f(n)}(b)(\bmod{\, n})$ for different
arithmetic functions $f$.
 \end{remark}
\section{Weak Carmichael numbers}

\subsection{Sum of powers of coprime residues of $n$}
The {\it Euler totient function} $\varphi(n)$ is defined as equal to the 
number of positive integers less than $n$ which are relatively prime
to $n$. Each of these $\varphi(n)$ integers is called a {\it totative}
(or ``{\it totitive}") of $n$ (see \cite[Section 3.4, p. 242]{sc} where this notion 
is attributed to J.J. Sylvester). Let $t(n)$ denote the 
set of all totatives of $n$, i.e., 
$t(n)=\{j\in\Bbb N:\, 1\le j<n,\gcd(j,n)=1\}$. Given any fixed nonnegative
integer $k$, in 1850 A. Thacker (see \cite[p. 242]{sc})
introduced the function $\varphi_k(n)$ defined as
  \begin{equation}
\varphi_k(n)=\sum_{t\in t(n)}t^k
  \end{equation}
where the summation ranges over all totatives $t$ of $n$ (in addition,
we define $\varphi_k(1)=0$ for all $k$). 
Notice that $\varphi_0(n)=\varphi(n)$ and 
 $\varphi_k(n)=S_k(n)$ holds if and only if $n=1$ or $n$ is a prime number.

The following recurrence relation for the functions $\varphi_k(n)$ 
was established in 1857 by J. Liouville (cf. \cite[p. 243]{sc}):
  $$
\sum_{d\mid n}{n\choose d}^k\varphi_k(d)=S_k(n+1):=1^k+2^k+\cdots +n^k
 $$
which for $k=0$ reduces to {\it Gauss' formula} $\sum_{d\mid n}\varphi(d)=n$.
Furthermore, in 1985 
P.S. Bruckman \cite{br} established an explicit Bernoulli's-like formula 
for the {\it Dirichlet series} of $\varphi_k(n)$ defined as 
$f_k(s)=\sum_{k=1}^{\infty}\varphi_k(n)/n^s$ (there  $\varphi_k(n)$
is called generalized Euler function). 
Quite recently, in \cite[Corollary 1.9]{me2} the author of this article 
proved for all $k\ge 1$  and $n\ge 2$ the following recurrence relation involving the functions 
$\varphi_k(n)$. 
  \begin{equation*}
\sum_{i=0}^{2k-1}(-1)^i{2k-1\choose i}2^{2k-1-i}n^i\varphi_{2k-1-i}(n)=0.
 \end{equation*}
\subsection{Weak Carmichael numbers}
Inspired by the  previous definitions, results, and considerations 
we give the following definition.
  
   \begin{definition}
A composite positive integer $n$ is said to be a 
{\it weak Carmichael number} if 
  \begin{equation}
 \sum_{\gcd(k,n)=1\atop 1\le k\le n-1}k^{n-1}\equiv \varphi(n) \pmod{n},
  \end{equation}
where the summation ranges over all $k$ such that $1\le k\le n-1$ 
and $\gcd(k,n)=1$.
 \end{definition}
From the above definition we see that each Carmichael number is
also a weak Carmichael number; hence the name.
This together with the mentioned  result
that the set  of Carmichael numbers is infinite implies the  
 following fact.

\begin{proposition}
There are infinitely many weak Carmichael numbers.
\end{proposition}

The following  characterization of weak Carmichael numbers
may be useful for computational purposes.

     \begin{proposition}
Every weak Carmichael number is odd. Furthermore,  
an odd composite positive integer $n$ is a weak Carmichael number if and only 
if
  \begin{equation}
2\sum_{i=1}^{\varphi(n)/2}r_i^{n-1}\equiv \varphi(n)\pmod{n},
  \end{equation}
where $r_1<r_2<\cdots <r_{\varphi(n)}$ are all reduced residues modulo $n$.
  \end{proposition}

As noticed above, the  results, definitions and conjectures 
in this article are mainly based on  Theorem 1.5 
(a result of Borwein and Wong \cite[Corollary 8]{bw}) 
which in terms of weak Carmichael numbers can be reformulated as 
the following Korselt's type criterion for 
characterizing weak Carmichael numbers.
  \begin{theorem}
Let $n=p_1^{e_1}p_2^{e_2}\cdots p_s^{e_s}$ be a composite integer,
where $p_1,p_2,\ldots ,p_s$ are distinct odd primes and 
$e_1,e_2,\ldots ,e_s$ are positive integers. Then $n$ is a weak Carmichael 
number if and only $p_i-1\mid n-1$ for every $i=1,2,\ldots,s$. 
 \end{theorem}
\begin{remark}  Any integer greater than 1 and satisfying 
the congruence (2.2) is called in \cite{bw}) a {\it generalized Carmichael 
number}. Therefore, by Definition 2.1, the set of all generalized Carmichael 
numbers is a union of the set 
of weak Carmichael numbers and the set of all primes. 
The following result of  E. Wong (\cite[p. 17, Subsection 2.5.3]{w} where 
weak Carmichael numbers are called {\it pseudo-Carmichael numbers}) is 
immediate by Euler totient theorem and it establish   the fact that there are numerous 
weak Carmichael numbers that are not prime powers nor Carmichael numbers. 
   \end{remark}

Here, as always in the sequel, ${\rm lcm}(\cdot)$  will denote
the {\it least common multiple function}.
 \begin{proposition}
Let $p_1,p_2,\ldots,p_s$ be distinct primes such that 
$p_i-1\not\equiv 0 (\bmod{\, p_j})$ for each pair of indices $i,j$
with $1\le i\not= j\le s$. For all $j=1,2,\ldots,s$ put 
$e_i={\rm lcm}_{1\le j\le s\atop j\not= i}\varphi(p_j)$. Then 
any number of the form 
$p_1^{k_1e_1}p_2^{k_2e_2}\cdots p_s^{k_se_s}$
with $k_i\ge 1$, is a weak Carmichael number. Conversely, 
if $n$ is a weak Carmichael number with prime factors $p_1,p_2,\ldots,p_s$, 
then $p_i-1\not\equiv 0 (\bmod{\, p_j})$ for each pair of indices $i,j$
with $1\le i\not= j\le s$.
 \end{proposition}

  \begin{definition}
Let $n\ge 3$ be any odd positive integer with a prime factorization
$n=p_1^{e_1}p_2^{e_2}\cdots p_s^{e_s}$.  
 Then the function $c_{w}(n)$ of $n$ is defined as  
   \begin{equation}
c_{w}(n)={\rm lcm}(p_1-1,p_2-1,\ldots,p_s-1).
  \end{equation}
 \end{definition}

From the above  definition, the definition of Carmichael function 
$\lambda(n)$ and its property that 
$p-1=\lambda(p)\mid \lambda(p^e)$ for any odd prime $p$ and 
$e\ge 2$, we immediately obtain the following result.
  \begin{proposition}
For each odd positive integer $n$, $c_{w}(n)\mid \lambda(n)$.
Therefore,  for such a $n$ we have $c_{w}(n)\le \lambda(n)$.
 \end{proposition}

Using Euler totient theorem, Theorem 2.4 easily yields 
the following result which gives a possibility  for the construction of 
weak Carmichael numbers via Carmichael numbers.
  \begin{proposition}
Let $n=p_1p_2\cdots p_s$ be an arbitrary Carmichael number. 
For any fixed $i\in\{1,2,\ldots,s\}$ let $d_i$ be 
a smallest positive divisor  of 
$\varphi(c_w(n/p_i))$ with 
$c_w(n/p_i):=\prod_{1\le j\le l\atop j\not= i}(p_j-1)$,
such that $p_i^{d_i}\equiv 1(\bmod{\,c_w(n/p_i)})$. 
Then $np_i^{md_i}$ is a weak Carmichael number for every positive 
integer $m$. 
 \end{proposition}
 \begin{example} 
Consider the smallest Carmichael number $561=3\cdot 11\cdot 17$. 
Then $c_w(561/3)={\rm lcm}(10,16)=80$, 
$c_w(561/11)={\rm lcm}(2,16)=16$ and $c_w(561/17)={\rm lcm}(2,10)=10$,
 and $d_1=4$, $d_2=8$, and $d_3=4$ are smallest integers
 for which $3^{d_1}\equiv 1(\bmod{\,80})$, $11^{d_2}\equiv 1(\bmod{\,16})$
and $17^{d_3}\equiv 1(\bmod{\,10})$, respectively.
This by Proposition 2.9 shows that $3^{4m+1}\cdot 11\cdot 17$, 
$3\cdot 11^{8m+1}\cdot 17$ and $3\cdot 11\cdot 17^{4m+1}$
are  weak Carmichael numbers for every positive integer $m$ 
(the smallest such a number $3^5\cdot 11\cdot 17=45441$ occurs in 
Table 1 as a smallest weak Carmichael number  described in Proposition 2.9). 
Similarly, regarding related values $d_1$  for a 
smallest prime divisor of the  next  four Carmichael numbers
 $1105, 1729, 2465$ and $2821$ (see Table 1),
we respectively obtain the following associated sequences for weak Carmichael 
numbers: $5^{4m+1}\cdot 13\cdot 17$, 
$7^{6m+1}\cdot 13\cdot 19$, $5^{12m+1}\cdot 17\cdot 29$ and 
$7^{4m+1}\cdot 13\cdot 31$ with  $m\ge 1$.
 \end{example}
 \begin{remark} 
As noticed above, in 1939, Chernick \cite{ch} gave a simple method to obtain
Carmichael numbers with three prime factors.
The distribution of primes with three prime factors 
has been studied in 1997 by R. Balasubramanian and S.V. Nagaraj \cite{bn},
who showed that the number of such Carmichael numbers up to $x$ is at most 
$O(x^{5/(14+o(1))})$. If $n=pqr$ is a Carmichael number then we have $p-1=da$, 
$q-1=db$ and $r-1=dc$ where $a,b$ and $c$ are coprime and $dabc\mid n-1$.
The Chernick form $n=pqr=(6m+1)(12m+1)(18m+1)$ is a special 
case of the form 
$$
n=pqr=(am+1)(bm+1)(cm+1)
$$ 
with $a<b<c,$ where $a,b$ and $c$ are relatively  prime in pairs. 
Namely, the  case $a=1$, $b=2$, 
$c=3$, leading to $d\equiv 0(\bmod{\,6})$. We see that  most values 
$(amb,c)$ will lead to a possible congruence for $d$ modulo $abc$,
whose smallest solution may be expected to be of the same order as $abc$.
As shown in \cite[the congruence (5)]{ch}
in Ore's book \cite[Ch. 14]{or}, $m=m_0+tabc$ with 
$t=1,2,3,\ldots$, where $m_0$ is the solution to the 
linear congruence
  \begin{equation}
m_0(ab+ac+bc)\equiv -(a+b+c)\pmod{abc}.
 \end{equation}
Thus, for  given $a,b,c$ it is easy to find all allowable values 
of $m$. All that remains is to test the three components for primality 
for each allowable $m$. In this way a ``family" of Carmichael numbers 
is found corresponding to triplets $(a,b,c)$. In 
\cite[Section 5, Table 2]{du} 
H. Dubner reported that the counts of $(1,a,b)$ are about $64.4\%$
of the corresponding Carmichael numbers with three prime factors less 
than $10^n$ for a wide range of $n$. Moreover, the counts of
$(1,a,b)$ are about $2.2\%$ of such Carmichael numbers. 
 
However, it is not yet known whether there are infinitely many 
Carmichael numbers of Chernick form, although this would folow from the more 
general conjecture of Dickson \cite{di}.
In 2002 H. Dubner \cite{du} tabulated the counts of Carmichael numbers
 of Chernick form up to $10^n$ for each $n<42$.  
Up to $10^{12}$ and $10^{18}$ there are 
respectively 1000 and 35586  with three prime factors 
(see \cite[Table 2]{du}).
Between  these 1000 (resp. 35585) Carmichael numbers, 25 (resp. 783)
numbers correspond to the Chernick form with related triplets $(a,b,c)=(1,2,3)$ 
(see \cite[Table 1]{du}).
 \end{remark}
 \begin{example} Here we present a simple way for constructing 
 weak Carmicha-\
el numbers with four prime factors using the Chernick form
of product $(6m+1)(12m+1)(18m+1)$. Consider the extended Chernick product
in the form 
   \begin{equation}
C(m; d,l):=(6m+1)(12m+1)(18m+1)\left(\frac{36m}{d}+1\right)^l,
   \end{equation}
with $d\mid 36m$ and some $l\ge 1$. 
Then under the assumptions that $p=6m+1,q=12m+1$  and  $r=18m+1$
are primes, a routine calculation shows that $C(m; d,l)$ is a 
weak Carmichael numbers  
with four prime factors if and only if $w(m,d):=36m/d+1$ is a prime 
different from $p$, $q$ and $r$ such that 
$(36m/d+1)^l\equiv 1(\bmod{\,6m})$. In particular, 
for a given $m$, possible values $d=1,4,9,12,18,36$ 
respectively give the following values for   $s$: 
$36m+1, 9m+1, 4m+1, 2m+1, m+1$. For example, Chernick \cite[p. 271]{ch} 
observed that the integers $C(m):=(6m+1)(12m+1)(18m+1)$ are 
Carmichael numbers for $m\in\{1,6,35,45,51,55,56,100,121\}$. For a fixed $m\ge 1$, denote by 
$W_m$ the set of all odd primes in the set $\{36m/d+1:\, d\mid 36\}$. 
Then $W_1=\{3,5,7\}$,  $W_6=\{7,13,19,37,217\}$,  
$W_{35}=\{7,13,19\}$,  $W_{45}=\{71\}$,  $W_{51}=\{103\}$,  
$W_{55}=\Phi$,   $W_{56}=\{13,2017\}$,   
$W_{100}=\{401\}$ and  $W_{121}=\{4357\}$.   
Then every $w\in W_m$ for some $m\in\{1,6,35,45,51,55,56,100,121\}$
arise a set of weak Carmichael numbers of the form  $C(m; d,l)$ given by 
(2.6), where 
$l$ must satisfy the congruence $w^l\equiv 1(\bmod{\,6m})$.
For example, assuming $w=13\in W_{35}$, we arrived to the 
set of weak Carmichael numbers of the form $211\cdot 421\cdot 631\cdot 13^l$
with $l$ such that $13^l\equiv 1(\bmod{\,210})$. Using the fact that 
$\varphi(210)=48$,  we can easily verify that $l_0=4$ is the smallest 
value of $l$ satisfying the previous congruence. Consequently,
each integer of the form $211\cdot 421\cdot 631\cdot 13^{4u}$  
with $u=1,2,\ldots$ is a weak Carmichael number.
   \end{example}
   \begin{definition}
A weak Carmichael number which can be obtained from certain Carmichael number
in  the manner described in Proposition 2.9 is called 
  a {\it  Carmichael like number}.  
  \end{definition}
 \begin{remark}
From Table 1 we see that there exist many weak Carmichael numbers of the form 
$n=p_1\cdots p_{k-1}p_k^f$ with some $k\in\{3,4\}$ and $f\ge 2$,
which are not Carmichael numbers. For example, from Table 1 we see that 
$8625=3\cdot 5^3\cdot 23$ is the smallest such number, and 
the smallest such numbers with four distinct prime factors 
is $54145=5\cdot 7^2\cdot 13\cdot 17$.
  \end{remark}
In terms of the function $c_{w}(n)$, 
Theorem 2.4 can be reformulated as follows.

\vspace{1mm}

\noindent{\bf Theorem 2.4'.} {\it Let $n=p_1^{e_1}p_2^{e_2}\cdots p_s^{e_s}$ 
be a composite integer, where $p_1,p_2,\ldots ,p_s$ are odd distinct primes and 
$e_1,e_2,\ldots ,e_s$ are positive integers. 
Then $n$ is a weak Carmichael number 
if and only $c_{w}(n)\mid n-1$.}
\vspace{2mm}
 
As an immediate consequence of Theorem 2.4, we establish a surprising result 
that summing  all $\varphi(n)$ congruences $a^{n-1}\equiv 1(\bmod{\,n})$ 
over $1\le a\le n-1$ with $\gcd(a,n)=1$, we obtain the congruence 
which  characterizes Carmichael numbers under the assumption that $n$ 
is a square-free integer. 
  \begin{theorem}
Let $n>1$ be a square-free positive integer. Then $n$ is a 
Carmichael number if and only if 
  \begin{equation}
 \sum_{\gcd(k,n)=1\atop 1\le k\le n-1}k^{n-1}\equiv \varphi(n) \pmod{n},
 \end{equation}
 \end{theorem}
Using the well known fact that every Carmichael number is square-free,
as a consequence of Theorem 2.15, we obtain the following simple characterization 
of Carmichael numbers.
 \begin{corollary} A composite positive integer $n$ is a Carmichael number if 
and only if the following conditions are satisfied.

$\,\, (i)$ $n$ is square-free and 

$(ii)$ $\displaystyle\sum_{\gcd(k,n)=1\atop 1\le k\le n-1}k^{n-1}\equiv 
\varphi(n) \pmod{n}$.
  \end{corollary}
Recall that the {\it M\"{o}bius $\mu$-function} is defined so that $\mu(1)=1$, 
$\mu(n)=(-1)^s$ if $n$ is a product of $s$ distinct 
primes, and $\mu(n)=0$ if $n$ is divisible by the square of a prime.
Then the following consequence of Theorem 2.4 gives a characterization of 
weak Carmichael numbers that are  not  Carmichael numbers. 
 \begin{corollary} An integer $n>1$ is a weak Carmichael number which 
is not a Carmichael number if and only if 
  $$
\sum_{\gcd(k,n)=1\atop 1\le k\le n-1}k^{n-1}\equiv \varphi(n)+
\mu(n) \pmod{n}.
  $$
 \end{corollary}
Theorem 2.4 immediately gives the following result which was 
also observed in \cite{bw}, and also directly proved in \cite[Lemma 1]{gm3}. 
 \begin{proposition} 
Every power $p^e$ of any odd prime $p$ with $e\ge 2$ is a weak 
Carmichael number.
 \end{proposition}
   \begin{corollary}
If $n$ is a weak Carmichael number, then every power 
$n^e$ of $n$ with $e=2,3,\ldots$ is also a weak Carmichael number.
In particular, such a power of any Carmichael number is also 
a weak Carmichael number.
     \end{corollary}
Theorem 2.4 and the well known fact that every Carmichael number has at least 
three distinct prime factors imply the following result. 
   \begin{corollary}
Let $n=pq$ be a product of distinct odd primes $p$ and $q$. 
Then $n$ is not a weak Carmichael number.
    \end{corollary}
\begin{remark} Recall that in Section 3 we give a direct proof 
of Corollary 2.20.
\end{remark}

Proposition 2.18 shows that  weak Carmichael numbers
appear to be more numerous than the Carmichael numbers,
which can be expressed as follows.
  \begin{corollary} 
Let $C(x)$ and  $C_w(x)$ be the numbers of Carmichael numbers and weak 
Carmichael numbers in the interval $[1,x]$, respectively. Then 
  $$
\lim_{x\to\infty}(C_w(x)-C(x))=+\infty.
  $$
 \end{corollary}
  \begin{remark}
Obviously, Corollary 2.17 may be very significant for compuational search
of Carmichael numbers. Namely, in order to examine whether 
a given non-square positive integer $n$ is a Carmichael number, it
is sufficient to verify only one congruence modulo $n$.
However, for related faster compuations may be useful the following 
charaterization  of Carmichael numbers which immediately 
follows from Corollary 2.17 and the fact that 
$\varphi(p_1p_2\ldots p_k)=(p_1-1)(p_2-2)\cdots (p_k-1)$.
  \end{remark}

  \begin{corollary}
Let $n=p_1p_2\cdots p_s$ be a composite positive integer, where 
$p_1,p_2,\ldots , p_s$ are distinct primes. Let $l_i$ be residues
of $n-1$ modulo $p_i$ with $i=1,2,\ldots,s$.
Then $n$ is a  Carmichael number if 
and only if  the following $k$ congruences are satified:
 \begin{equation}
 \sum_{\gcd(k,n)=1\atop 1\le k\le l-1}k^{l_i}\equiv 
-\frac{1}{p_i-1}\prod_{j=1}^{s}(p_j-1)\pmod{p_i}, \quad i=1,2,\ldots , s.
  \end{equation}
  \end{corollary}
 Proposition 2.6 motivates the following definition.
  \begin{definition}
A weak Carmichael number $n$ is called a {\it primitive weak Carmichael 
number} if $n\not= m^f$ for every weak Carmichael numbers $m$
and all integers $f\ge 2$.
 \end{definition}
  \begin{remark}
Clearly, each weak Carmichael number which is not a power of some 
integer is also a primitive weak Carmichael number. In particular,
this is true for all  Carmichael numbers. However, 
there are  primitive weak Carmichael numbers which are powers of 
some integers. For example, Corollary 2.20 implies 
that  every  weak Carmichael number of the form $p^2q^2$ 
is a  primitive weak Carmichael number. From Table 1 we read  the 
following perfect squares of product of distinct primes: 
225, 1225,  8281 and 14161. We also see from Table 1 
that the numbers $2025=3^4\cdot 5^2$, $18225=3^6\cdot 5^2$ 
are  primitive weak Carmichael numbers, but 
2025 is not a primitive weak Carmichael number
(in view of the fact that its square root 45 is a weak Carmichael number).
Table 1 also shows that there are primitive Carmichael numbers which are
aquare of non-square integers (for example, 
$1071225=(3^2\cdot 5\cdot 23)^2$).

Furthermore, in view of Definition 2.25, Proposition 2.18 we have the 
following result.
 \end{remark}
   \begin{corollary}
If $p$ is an odd prime, then $p^f$ is a primitive weak Carmichael number
if and only if $f$ is a prime.
    \end{corollary}
The facts that there are infinitely many Carmichael numbers 
and that every Carmichael number is a weak Carmichael number
yield the following result.
   \begin{corollary}
There are infinitely many primitive weak Carmichael numbers
which are not  prime powers.
    \end{corollary}
Applying the congruence (2.2), we find via {\tt Mathematica} 8 the following 
table of weak Carmichael numbers up to $25000$ and their 
factorizations. In this table Carmichael numbers are written in boldface,
while prime powers are written in italic face.
The notion of indices of weak Carmichael numbers 
which are less than 30000, given in Table 2, are described in Subsection 2.6.
  \vspace{2mm}
\begin{center}
{\bf Table 1.} Weak Carmichael numbers up to $25000$
{\small
\begin{tabular}{|c|c|c|c|}\hline
${\it 9 = 3^2}$ & ${\it  25=5^2}$ &  ${\it  27=3^3}$ & $45=3^2\cdot 5$\\
 ${\it  49=7^2}$ & ${\it 81=3^4}$ & ${\it 121 =11^2}$ & ${\it 125 =5^3}$ \\
   ${\it 169=13^2}$ &  $225= 3^2\cdot 5^2$ & ${\it  243 =3^5}$ & 
${\it 289 =17^2}$\\
  $325 =5^2 \cdot13 $ &  ${\it 343= 7^3}$&  ${\it 361= 19^2}$&
$405 =3^4\cdot 5$\\
 ${\it 529 = 23^2}$&  ${\bf 561 =3 \cdot 11 \cdot 17}$&
  ${\it 625= 5^4}$&  $637=7^2\cdot 13$ \\
${\it 729 =3^6}$ & ${\it 841 =29^2}$ &  $891 = 3^4\cdot 11 $ &  
${\it 961= 31^2}$\\
  ${\bf 1105= 5\cdot13\cdot 17}$ & $1125=3^2\cdot 5^3$ & 
$1225 =5^2\cdot 7^2$& ${\it 1331 =11^3}$\\
  ${\it 1369=37^2} $&  $1377=3^4\cdot 17 $ & ${\it 1681=41^2}$ & 
${\bf 1729 =7\cdot 13\cdot 19}$\\
 ${\it 1849 =43^2}$& $2025=3^4 \cdot 5^2 $&  ${\it 2187= 3^7}$ & ${\it 2197=13^3}$ \\ 
${\it 2209=47^2}$ & ${\it 2401=7^4}$ & ${\bf 2465=5\cdot 17\cdot 29}$
& ${\it 2809=53^2}$\\
  ${\bf 2821=7\cdot 13\cdot 31}$ & 
${\it 3125=5^5}$& $3321=3^4\cdot 41$ &  ${\it 3481=59^2}$\\
 $3645=3^6\cdot 5$ & ${\it 3721=61^2}$&  $3751=11^2\cdot 31$ & 
$3825=3^2\cdot 5^2\cdot 17$ \\
 $4225=5^2\cdot 13^2$&  ${\it 4489=67^2}$ & ${\it 4913=17^3}$ & 
$4961=11^2\cdot  41$\\
  ${\it 5041=71^2}$ & ${\it 5329=73^2}$ & $5589=3^5\cdot 23$ 
& $5625=3^2\cdot 5^4$\\
  ${\it 6241=79^2}$ & $6517=7^3\cdot 19$ & $6525=3^2\cdot 5^2\cdot 29$& 
${\it 6561=3^8}$\\
 ${\bf 6601=7\cdot 23\cdot 41}$ & ${\it 6859=19^3}$ & 
${\it 6889=83^2}$ & $7381=11^2\cdot 61$\\
 ${\it 7921=89^2}$ & $8125=5^4\cdot 13$ & $8281=7^2\cdot 13^2$ & $8625=3\cdot
5^3\cdot 23$\\
 ${\bf 8911=7\cdot 19\cdot 67}$ & ${\it 9409=97^2}$ & $9801=3^4\cdot 11^2$ & 
$10125=3^4\cdot 5^3$\\
 ${\it 10201=101^2}$ & ${\bf 10585=5\cdot 29\cdot 73}$ & 
${\it 10609=103^2}$ &  $10625=5^4\cdot 17$ \\
 ${\it 11449=107^2}$ & ${\it 11881=109^2}$ & 
$12025=5^2\cdot 13\cdot 37$ & ${\it 12167=23^3}$ \\
 ${\it 12769=113^2}$ & $13357=19^2\cdot 37$ & 
$13833=3^2\cdot 29\cdot 53$& $14161=7^2\cdot 17^2$ \\
 ${\it 14641=11^4}$ & ${\it 15625=5^6}$ & ${\bf 15841=7\cdot 31\cdot 73}$ &
 $15925=5^2\cdot 7^2\cdot 13$\\
 ${\it 16129=127^2}$ &  ${\it 16807=7^5}$ & 
${\it 17161=131^2}$ & $18225=3^6\cdot 5^2$\\ 
 ${\it 18769=137^2}$ & ${\it 19321=139^2}$ & 
${\it 19683=3^9}$ & $21141=3^6\cdot 29$ \\
${\it 22201=149^2}$  & ${\it 22801=151^2}$ & $23409=3^2\cdot 5\cdot 23^2$ 
&$23805=3^2\cdot 5\cdot 23^2$\\
 ${\it 24389=29^3}$ & ${\it 24649=157^2}$ &$$& $$\\\hline
  \end{tabular}}
 \end{center}
  \begin{remark}
Table 1 shows that there are $102$ weak Carmichael numbers less than
$25000$, and between them there are 9 Carmichael numbers,  57
odd  prime powers, and 36 other composite numbers. 
 Recall that in 2006 R.G.E. Pinch \cite{pi2} reported that there are  
1401644 Carmichael numbers up to $10^{18}$ (also see \cite{pi1}
for a search of total 105212 Carmichael numbers up to $10^{15}$).
Notice that $1401644\approx 1.4\times (10^{18})^{1/3}$.
  \end{remark}

\vfill\eject
\begin{center}
{\bf Table 2.} Weak Carmichael numbers up to $2\times 10^6$
that are not prime powers
\qquad\qquad\qquad\qquad\qquad\qquad\qquad\qquad
{\small
  \begin{tabular}{|c|c|c|}\hline
$45_{8}=3^2\cdot 5$  & $225_8=3^2\cdot 5^2$ & $325_{48} =5^2 \cdot13$\\ 
$405_8 =3^4\cdot 5$ &  ${\bf 561_{320} =3 \cdot 11 \cdot 17}$ & 
$637_{72}=7^2\cdot 13$ \\ 
$891_{20}= 3^4\cdot 11 $ & 
${\bf 1105_{768}= 5\cdot13\cdot 17}$ & $1125_8=3^2\cdot 5^3$\\
$1225_{24}=5^2\cdot 7^2$ & $1377_{32}=3^4\cdot 17$ &
 ${\bf 1729_{1296}=7\cdot 13\cdot 19}$\\
 ${\it 2025_{8}=3^4 \cdot 5^2} $& ${\bf 2465_{1792}=5\cdot 17\cdot 29}$ 
& ${\bf 2821_{2160}=7\cdot 13\cdot 31}$\\
 $3321_{80}=3^4\cdot 41$ & $3645_{8}=3^6\cdot 5$ &
 $3751_{300}=11^2\cdot 31$\\
 $3825_{128}=3^2\cdot 5^2\cdot 17$ &$4225_{48}=5^2\cdot 13^2$ &
 $4961_{400}=11^2\cdot  41$ \\
 $5589_{44}=3^5\cdot 23$& $5625_{8}=3^2\cdot 5^4$ & 
$6517_{108}=7^3\cdot 19$ \\
 $6525_{224}=3^2\cdot 5^2\cdot 29$ &
 ${\bf 6601_{5280}=7\cdot 23\cdot 41}$& $7381_{600}=11^2\cdot 61$\\ 
$8125_{48}=5^4\cdot 13$ & $8281_{72}=7^2\cdot 13^2$ & 
 $8625_{176}=3\cdot 5^3\cdot 23$\\
 ${\bf 8911_{7128}=7\cdot 19\cdot 67}$ & $9801_{20}=3^4\cdot 11^2$ & 
$10125_{8}=3^4\cdot 5^3$ \\
 ${\bf 10585_{8064}=5\cdot 29\cdot 73}$ &
 $10625_{64}=5^4\cdot 17$ & $12025_{1728}=5^2\cdot 13\cdot 37$ \\
 $13357_{648}=19^2\cdot 37$&
$13833_{2912}=3^2\cdot 29\cdot 53$ &
 $14161_{96}=7^2\cdot 17^2$ \\ 
${\bf 15841_{12960}=7\cdot 31\cdot 73}$ & 
$15925_{288}=5^2\cdot 7^2\cdot 13$ &
 $18225_{8}=3^6\cdot 5^2$\\  $21141_{56}=3^6\cdot 29$ & 
$23409_{32}=3^4\cdot 17^2$ &
$23805_{176}=3^2\cdot 5\cdot 23^2$ \\
$25425_{896}=3^2\cdot 5^2\cdot 113$ &
$26353_{1296}=19^2\cdot 73$ &
$28033_{1536}=17^2\cdot 97$\\  
$28125_{8}=3^2\cdot 5^5$ & 
${\bf 29341_{25920}=13\cdot 37\cdot 61}$ &
 $30625_{}^{}=5^4\cdot 7^2$ \\
$31213_{}^{}=7^4\cdot 13$ & $32805_{}^{}=3^8\cdot 5$ & 
$33125_{}^{}=5^4\cdot 53$ \\
 $35425_{}^{}=5^2\cdot 13\cdot 109$ &
$35443_{}^{}=23^2\cdot 67$ &
 $38637_{}^{}=3^6\cdot 53$ \\
 ${\bf 41041=7\cdot 11\cdot 13\cdot 41}$ &
 $ 41125_{}^{}=5^3\cdot 7\cdot 47$ & 
 $ 45325_{}^{}=5^2\cdot 7^2\cdot 37$ \\ 
 $45441_{}^{}=3^5\cdot 11\cdot 17 $ &
 $ {\bf 46657=13\cdot 37\cdot 97}$ & 
 $ 47081_{}^{}=23^2\cdot 89$ \\ 
 $ 47125_{}^{}=5^3\cdot 13\cdot 29$ & 
 ${\it 50625_{}^{}=3^4\cdot 5^4}$& 
${\bf 52633=7\cdot 73\cdot 103}$\\
 $ 54145_{}^{}=5\cdot 7^2\cdot 13\cdot 17$ &
 $ 54925_{}^{}=5^2\cdot 13^3$ &
 $ 58621_{}^{}=31^2\cdot 61$ \\
 $60025_{}^{}=5^2\cdot 7^4$ &
 ${\bf 62745=3\cdot 5\cdot 47\cdot 89}$ &
${\bf 63973=7\cdot 13\cdot 19\cdot 37}$ \\
$65025_{}^{}=3^2\cdot 5^2\cdot 17^2$ &
$65341_{}^{}=19^2\cdot 181$ &
$72171_{}^{}=3^8\cdot 11$ \\
$74431_{}^{}=7^4\cdot 31$ &
 ${\bf 75361=11\cdot 13\cdot 17\cdot 31}$ &
 $78625_{}^{}=5^3\cdot 17\cdot 37$ \\  
$81289_{}^{}=13^3\cdot 37$ & 
$ 83125=5^4\cdot 7\cdot 19$ &
$89425_{}^{}=5^2\cdot 7^2\cdot 73$ \\
$91125_{}^{}=3^6\cdot 5^3$ &
 $94501_{}^{}=11^3\cdot 71$ &
 $98125_{}^{}=5^4\cdot 157$ \\
$99541_{}^{}=13^2\cdot 19\cdot 31$ & 
$99937_{}^{}=37^2\cdot 73$ &
${\bf 101101=7\cdot 11\cdot 13\cdot 101}$ \\
${\it 105625_{}^{}=5^4\cdot 13^2}$ &
 $106641_{}^{}=3^2\cdot 17^2\cdot 41$ &  
$107653_{}^{}=7^2\cdot 13^3$ \\ 
$107811_{}^{}=3^4\cdot 11^3$ &
$111537_{}^{}=3^8\cdot 17$ &
  ${\bf 115921=13\cdot 37\cdot 241}$ \\ 
$116281_{}^{}=11^2\cdot 31^2$ &
 $117325_{}^{}=5^2\cdot 13\cdot 19^2$ & 
$123823_{}^{}=7^3\cdot 19^2$ \\ 
${\bf 126217=7\cdot 13\cdot 19\cdot 73}$& 
$128547_{}^{}=3^5\cdot 23^2$ & 
$134113_{}^{}=7^3\cdot 17\cdot 23$ \\
$136161_{}^{}=3^4\cdot 41^2$ &
 $140625_{}^{}=3^2\cdot 5^6$ & 
$142129_{}^{}=13^2\cdot 29^2$ \\
$146461_{}^{}=7^4\cdot 61$ & 
${\bf 162401_{}^{}=17\cdot 41\cdot 233}$ & 
${\it 164025_{}^{}=3^8\cdot 5^2}$\\
${\bf 172081=7\cdot 13\cdot 31\cdot 61}$&
$177331_{}^{}=7^3\cdot 11\cdot 47$ & 
$180225_{}^{}=3^4\cdot 5^2\cdot 89$ \\ 
$180625_{}^{}=5^4\cdot 17^2$ &
$187461_{}^{}=3^3\cdot 53\cdot 131$ & 
${\bf 188461_{}^{}=7\cdot 13\cdot 19\cdot 109}$ \\ 
$189225_{}^{}=3^2\cdot 5^2\cdot 29^2$ & 
$195625_{}^{}=5^4\cdot 313$ & 
$203125_{}^{}=5^6\cdot 13$ \\ 
$203401_{}^{}=11^2\cdot 41^2$ &
$203841_{}^{}=3^2\cdot 11\cdot 29\cdot 71$ & 
$207025_{}^{}=5^2\cdot 7^2\cdot 13^2$\\ 
$211141_{}^{}=7^2\cdot 31\cdot 139$ & 
$231601_{}^{}=31^2\cdot 241$&
$232897_{}^{}=7^4\cdot 97$\\ 
$236321_{}^{}=29^2\cdot 281$ & 
$239701_{}^{}=7\cdot 11^2\cdot 283$& 
$241129_{}^{}=7^3\cdot 19\cdot 37$\\
$251505_{}^{}=3^7\cdot 5\cdot 23$& 
${\bf 252601=41\cdot 61\cdot 101}$&
 $253125_{}^{}=3^4\cdot 5^5$\\
 $254221_{}^{}=11^3\cdot 191$& 
$261625_{}^{}=5^3\cdot 7\cdot 13\cdot 23$& 
$269001_{}^{}=3^8\cdot 41$\\ 
${\bf 278545=5\cdot 17\cdot 29\cdot 113}$&
$290521_{}^{}=7^4\cdot 11^2$&
 ${\bf 294409=37\cdot 73\cdot 109}$\\ 
$295245_{}^{}=3^{10}\cdot 5$& 
$306397_{}^{}=7^2\cdot 13^2\cdot 37$& 
$307051_{}^{}=47^2\cdot 139$\\
 $309825_{}^{}=3^6\cdot 5^2\cdot 17$& 
$312481_{}^{}=13^2\cdot 43^2$& 
$314721_{}^{}=3^2\cdot 11^2\cdot 17^2$\\ 
${\bf 314821_{}^{}=13\cdot 61\cdot 397}$&
$319345_{}^{}=5\cdot 13\cdot 17^3$&
$321201_{}^{}=3^2\cdot 89\cdot 401$\\
${\bf 334153_{}^{}=19\cdot 43\cdot 409}$&
$338031_{}^{}=3^2\cdot 23^2\cdot 71$&
${\bf 340561=13\cdot 17\cdot 23\cdot 67}$\\
$ 341341=7\cdot 11^2\cdot 13\cdot 31$&
$354061_{}^{}=29^2\cdot 421$&
$362551_{}^{}=7^4\cdot 151$\\
$378625_{}^{}=5^3\cdot 13\cdot 233$&
$388125_{}^{}=3^3\cdot 5^4\cdot 23$& 
$397953_{}^{}=3^4\cdot 17^3$\\ 
  \hline
   \end{tabular}}
 \end{center}
\vfill\eject

\begin{center}
{\bf Table 2.}  (Continued)
\qquad\qquad\qquad\qquad\qquad\qquad\qquad\qquad
{\small
\begin{tabular}{|c|c|c|}\hline

$398125_{}^{}=5^4\cdot 7^2\cdot 13$&
${\bf 399001=31\cdot 61\cdot 211}$& 
$401841_{}^{}=3^4\cdot 11^2\cdot 41 $\\ 
$405121_{}^{}=41^2\cdot 241$&
${\it 405769_{}^{}=7^4\cdot 13^2}$&
$409825_{}^{}=5^2\cdot 13^2\cdot 97$\\
${\bf 410041=41\cdot 73\cdot 137}$&
 $441013_{}^{}=53^2\cdot 157$&
$442225_{}^{}=5^2\cdot 7^2\cdot 19^2$\\
$444925_{}^{}=5^2\cdot 13\cdot 37^2$&
 ${\bf 449065=5\cdot 19\cdot 29\cdot 163}$&
$450241_{}^{}=11^2\cdot 61^2$  \\
 $453125_{}^{}=5^6\cdot 29$&
 $453871_{}^{}=11^4\cdot 31$&
$455625_{}^{}=3^6\cdot 5^4$\\
$462241_{}^{}=13\cdot 31^2\cdot 37$&
 $468391_{}^{}=7^2\cdot 11^2\cdot 79$&
$472361_{}^{}=41^2\cdot 281$\\ 
${\bf 488881=37\cdot 73\cdot 181}$& 
$494209_{}^{}=19^2\cdot 37^2$& 
$499681_{}^{}=7\cdot 13\cdot 17^2\cdot 19$\\ 
$501025_{}^{}=5^2\cdot 7^2\cdot 409$& 
$505141_{}^{}=7^2\cdot 13^2\cdot 61$& 
$511525_{}^{}=5^2\cdot 7\cdot 37\cdot 79$\\ 
${\bf 512461=31\cdot 61\cdot 271}$& 
${\bf 530881=13\cdot 97\cdot 421}$& 
$531505_{}^{}=5\cdot 13^2\cdot 17\cdot 37$\\ 
$544563_{}^{}=3^8\cdot 83$& 
${\bf 552721=13\cdot 17\cdot 41\cdot 61}$& 
$554625_{}^{}=3^2\cdot 5^3\cdot 17\cdot 29$\\ 
$561925_{}^{}=5^2\cdot 7\cdot 13^2\cdot 19$& 
$566401=11^2\cdot 31\cdot 151$& 
$578125_{}^{}=5^6\cdot 37$\\ 
$578641_{}^{}=7^4\cdot 241$& 
$595441_{}^{}=7\cdot 11^2\cdot 19\cdot 37$& 
$600281_{}^{}=11^4\cdot 41$\\ 
$604513_{}^{}=7^2\cdot 13^2\cdot 73$& 
$611893_{}^{}=47^2\cdot 277$& 
$613089_{}^{}=3^6\cdot 29^2$\\ 
$624169_{}^{}=7\cdot 13\cdot 19^3$&
$652257_{}^{}=3^2\cdot 23^2\cdot 137$& 
${\bf 656601=3\cdot 11\cdot 101\cdot 197}$\\ 
${\bf 658801=11\cdot 13\cdot 17\cdot 271}$ & 
${\bf 670033=7\cdot 13\cdot 37\cdot 199}$&
$690625_{}^{}=5^5\cdot 13\cdot 17$\\ 
$693889_{}^{}=7^4\cdot 17^2$& 
$695871_{}^{}=3^4\cdot 11^2\cdot 71$& 
$703125_{}^{}=3^2\cdot 5^7$\\ 
$713125_{}^{}=5^4\cdot 7\cdot 163$& 
$714025_{}^{}=5^2\cdot 13^4$& 
$717025_{}^{}=5^2\cdot 23\cdot 29\cdot 43$\\ 
${\bf 748657=7\cdot 13\cdot 19\cdot 433}$& 
$ 750541_{}^{}=11\cdot 31^2\cdot 71$& 
$ 750925=5^2\cdot 7^2\cdot 613$\\ 
$ 765625=5^6\cdot 7^2$& 
$767625=3\cdot 5^3\cdot 23\cdot 89$&
$777925=5^2\cdot 29^2\cdot 37$\\
$780325=5^2\cdot 7^4\cdot 13$&
$784225=5^2\cdot 13\cdot 19\cdot 127$&
$793117=13^3\cdot 19^2$\\
${\it 793881=3^8\cdot 11^2}$&
$803551=7^2\cdot 23^2\cdot 31$&
$808561=13\cdot 37\cdot 41^2$\\
$811073=59^2\cdot 233$&
$815121=3^2\cdot 41\cdot 47^2$&
$815425=5^2\cdot 13^2\cdot 193$\\
$820125=3^8\cdot 5^3$& 
${\bf 825265=5\cdot 7\cdot 17\cdot 19\cdot 73}$&
$838125=3^2\cdot 5^4\cdot 149$\\
${\bf 838201=7\cdot 13\cdot 61\cdot 151}$&
${\bf 852841=11\cdot 31\cdot 41\cdot 61}$&
$856087=43^2\cdot 463$\\
$856801=11^2\cdot 73\cdot 97$&
$860625=3^4\cdot 5^4\cdot 17$&
$877825=5^2\cdot 13\cdot 37\cdot 73$\\
$879217=53^2\cdot 313$&
$893101=11^4\cdot 61$&
$894691=7^2\cdot 19\cdot 31^2$\\
$943041=3\cdot 11\cdot 17\cdot 41^2$&
$965497_{}^{}=13^2\cdot 29\cdot 197$& 
$968485_{}^{}=5\cdot 7^2\cdot 59\cdot 67$\\
$970785_{}^{}=3^5\cdot 5\cdot 17\cdot 47$&
$979837_{}^{}=79^2\cdot 157$&
$989901_{}^{}=3^4\cdot 11^2\cdot 101$\\ 
${\bf 997633=7\cdot 13\cdot 19\cdot 577}$&
  &
\\\hline
$1002001=7^2\cdot 11^2\cdot 13^2$&
${\bf 1024651=19\cdot 199\cdot 271}$& 
$1030393=7\cdot 13^3\cdot 67$\\
${\bf 1033669=7\cdot 13\cdot 37\cdot 307}$&
${\bf 1050985=5\cdot 13\cdot 19\cdot 23\cdot 37}$&
$1063651=71^2\cdot 211$\\
$1071225=3^4\cdot 5^2\cdot 23^2$&
$1080801=3^2\cdot 29\cdot 41\cdot 101$&
${\bf 1082809=7\cdot 13\cdot 73\cdot 163}$\\
$1105425=3^2\cdot 5^2\cdot 17^3$&
$1140625=5^6\cdot 73$& 
${\bf 1152271=43\cdot 127\cdot 211}$\\
$1154881=7^4\cdot 13\cdot 37$&
$1165537=17^2\cdot 37\cdot 109$&
$1185921=3^4\cdot 11^4$\\
${\bf 1193221=31\cdot 61\cdot 631}$&
$1207845=3^3\cdot 5\cdot 23\cdot 389$&
$1214869=59^2\cdot 349$\\
$1221025=5^2\cdot 13^2\cdot 17^2$&
${\it 1265625=3^4\cdot 5^6}$& 
$1269621=3^3\cdot 59\cdot 797$\\
$1299961=13\cdot 19^2\cdot 277$&
$1321029=3^4\cdot 47\cdot 347$&
$1335961=11^2\cdot 61\cdot 181$\\
$1355121=3^2\cdot 17^2\cdot 521$&
$ 1357741=7^2\cdot 11^2\cdot 229$&
$ 1358127=3^{10}\cdot 23$\\
$1373125=5^4\cdot 13^3$&
$1399489=7^2\cdot 13^4$&
$1401841=7^3 \cdot 61 \cdot 67$\\
$ 1413721=29^2\cdot 41^2$&
$1416521=71^2 \cdot 281$&
$1439425=5^2\cdot 13\cdot 43\cdot 103$\\
$1443001=7^4\cdot 601$ &
${\bf 1461241=37\cdot 73\cdot 541}$&
$1468125=3^4\cdot 5^4\cdot 29$\\
$1476225=3^{10}\cdot 5^2$&
$1498861=7^{2}\cdot 13^2\cdot 181$&
${\it 1500625=5^4\cdot 7^4}$\\
$1506625=5^3 \cdot 17\cdot 709$&
$1529437=7^6\cdot 13$&
$1540081=17^2\cdot 73^2$\\
$1555009= 29^2\cdot 43^2$&
$ 1566891=3^3\cdot 131\cdot 443$&
${\bf 1569457=17 \cdot 19 \cdot 43  \cdot 113}$\\
$1610401=13^3 \cdot 733$&
$1615441=31^2\cdot 41^2$&
${\bf 1615681=23\cdot 199\cdot 353}$\\
$1653125=5^5\cdot 23^2$&
$1658385=3^2\cdot 5\cdot 137\cdot 269$&
$1677025=5^2\cdot 7^2\cdot 37^2$\\
$1710325=5^2\cdot 37\cdot 43^2 $&
$1741825=5^2\cdot 19^2\cdot 193$&
$1742221=13^4\cdot 61$\\
$1755625=5^4\cdot 53^2$&
${\bf 1773289=7\cdot 19\cdot 67\cdot 199}$&
$1815937=97^2\cdot 193$\\
${\bf 1857241=31\cdot 181\cdot 331}$&
${\it 1896129=3^8\cdot 17^2}$ &
${\bf 1909001=41\cdot 101\cdot 461}$\\
$1923769=19^2\cdot 73^2$&
$1935025=5^2\cdot 17\cdot 29\cdot 157$&
$1953433=79^2\cdot 313$\\\hline
   \end{tabular}}
 \end{center}
  \begin{remark}

Here, as always in the sequel, the Carmichael number(s) and 
weak Carmichael number(s) will be  often denoted by $CN$ and 
$WCN$, respectively. A computation via {\tt Mathematica 8}  shows that there are 
``numerous" weak Carmichael numbers  that are neither Carmichael numbers nor 
prime powers. In particular, Table 2 shows that up to $10^6$ there are 
$235$ weak Carmichael numbers which are not prime powers,  
and between them there are 43 Carmichael numbers.  Moreover, up to 
$2\times 10^6$  there are $298$ $WCN$ 
which are not prime powers,  and between them there are 
$55$ $CN$.
 \end{remark}

  \begin{remark}
It is known \cite[p. 338]{or} that a Carmichael number can be 
a product of two other Carmichael numbers; for example, such a number 
is 
$$
(7\cdot 13\cdot 19)(37\cdot 73\cdot 109)=1729\cdot 294409=509033161.
$$
It can be of interest to consider a related problem extended to the set 
of $WCN$ which are not prime powers (for example, $10125=45\times 225$,
$18225=45\times 405$ and $50625=45\times 1125$).
  \end{remark}
 \begin{example} 
Notice that it is easy to determine $WCN$ 
with two distinct prime factors. In \cite{bw} the authors observed that 
such numbers are all integers of the form $3^{2e}5^f$ for any $e,f\ge 1$,
and more generally, given any two odd primes $p<q$ with 
$q-1\not\equiv 0 (\bmod{\, p})$,  $p^{e\varphi(q-1)}q^{f\varphi(p-1)}$
is a $WCN$. For arbitrary given positive integers $e$ 
and $f$ such that $e\ge f$ and $e+f\ge 3$, denote by ${\mathcal C}_w(e,f)$
the set of all $WCN$ of the form $n=p^eq^f$ for 
some distinct odd primes $p$ and $q$. For any odd prime $p$ let 
${\mathcal C}_w(p;e,f)$ denote the set of all primes $q$  such that
$p^eq^f\in {\mathcal C}_w(e,f)$. 
Then by Theorem 2.4, $n$ is in ${\mathcal C}_w(e,f)$ if and only if 
$p-1\mid q^f-1$ and $q-1\mid p^e-1$, or equivalently
with  $p-1\mid q^f-1$ and $q-1\mid p^e-1$, respectively. In other words, for 
any given odd  prime $q$, a prime $p$ is in ${\mathcal C}_w(q;e,f)$   
if and only if $p-1\mid q^f-1$ and $q-1\mid p^e-1$.
For example, when $e=1$ and $f=2$, the above two conditions easily 
reduced to the condition of finding all divisors $d\ge 2$ of 
$p+1$ such that the number $q:=d(p-1)+1$ is a prime. Examining 
this condition for primes $p\in\{3,5,7,11,13,\ldots ,997\}$
(all 168 primes less than 1000), we find 
452 {\it WCN} of the form 
$p^2q$. We have verified also that into prime  factorizations
of these 452 numbers does not occur only 
primes 107, 317, 433 and 857 less than 1000,
while between other 164 these primes, each of the  primes 
$13,73,193,277,313,397,421,457,541,613,673,733,757$ and $787$
occur only as a non-square factor $q$ into {\it WCN} $p^2q$
(for example, for the first such number 13, $5^2\cdot 13$ 
is a $WCN$, and for the latest between them, 787, $263^2\cdot 787$ 
is a $WCN$).

 Since for a given $p$ and  divisors $d_1=2$ and $d_2=(p+1)/2$ of $p+1$, 
we have the candidates  $q_1=2p-1$ and $q_2=(p^2+1)/2$ for $q$, respectively.
In the first case, if $2p-1$ is also a prime, we obtain
that  $n_1=p^2(2p-1)$ is a {\it WCN}. In the second case, if  $(p^2+1)/2$ 
is a prime, then $p^2(p^2+1)/2$ is  a {\it WCN}. Notice that it 
was  conjectured that there are infinitely many pairs 
$(p,2p+1)$ such that both $p$ and $2p+1$ are primes 
(such a prime $p$ is called a {\it Sophie Germain prime}; AOO5384 in OEIS).
A computation shows that there are many pairs $(p,2p-1)$
such that both numbers $p$ and $2p-1$  are primes
(up to $10^3$, $10^4$, $10^5$, $10^6$, $10^7$ there are 
153, 1206, 9686, 82374 and 711033 such pairs, respectively,
while related numbers of Sophie Germain primes are 167, 1222, 9668,
82237 and 711154 respectively).
 Moreover, there are many triplets $(p,2p-1,(p^2+1)/2)$
such that the all  numbers $p,2p-1$ and $(p^2+1)/2$ are primes
(up to $10^3$, $10^4$, $10^5$, $10^6$, $10^7$ there are 
30, 180, 1113, 8029 and 58294  such triplets, respectively.) 

Similarly, if $e=3$ and $f=1$, then the corresponding conditions are 
equivalent to finding all divisors $d\ge 2$ of $p^2+p+1$ such that the number 
$q:=d(p-1)+1$ is a prime. For example, using this condition 
to the primes $p\in\{3,5,7,11,13\}$, we obtain the following 
three numbers in  ${\mathcal C}_w(3,1)$: $7^3\cdot 19=6517$, 
$11^3\cdot 71=94501$ and $11^3\cdot 191=254221$. 

A determination of some elements of ${\mathcal C}_w(2,2)$ consists in finding 
distinct odd primes $p$ and $q$ such that $p-1\mid q^2-1$ and 
$q-1\mid p^2-1$. Using this for $p\in\{3,5,7,11,13\}$, we get the following 
numbers in  ${\mathcal C}_w(2,2)$: $3^2\cdot 5^2=225$, 
$5^2\cdot 7^2=1225$, $5^2\cdot 13^2=4225$, $7^2\cdot 13^2=8281$,
$7^2\cdot 17^2=14161$, $11^2\cdot 31^2=116281$, $11^2\cdot 41^2=203401$,
$11^2\cdot 61^2=450241$, $13^2\cdot 29^2=142129$ and 
$13^2\cdot 43^2=312481$. 

For arbitrary pair $(e,f)$ of integers $e$ and $f$ with $1\le e\le f$ 
and $e+f\ge 3$, let ${\mathcal P}_w(e,f)$ be  a set defined as 
a a set of all pairs $(p,q)$ of distinct  primes $p$ and $q$ such that 
$p^eq^f\in {\mathcal C}_w(e,f)$. Since $p^e-1\mid p^{e'}-1$ 
whenever $e\mid e'$, it follows that for every such a pair $(e,e')$,
 ${\mathcal P}_w(e',f)\subseteq {\mathcal P}_w(e,f)$ holds. 
We conjecture that the converse statement is also true, that is, we have
  \end{example}
 \begin{conjecture}
If ${\mathcal P}_w(e',f')={\mathcal P}_w(e,f)$ then 
$f=f'$ and $e\mid e'$, or $e=e'$ and $f\mid f'$.
 \end{conjecture}
Furthermore, for the pair $(e,f)$ with $1\le e\le f$ 
and $e+f\ge 3$ let ${\mathcal Q}_w(e,f)$ be a set defined as
   \begin{equation*}\begin{split}
{\mathcal Q}_w(e,f)&=\{p:\, p\,\, {\rm is\,\, a\,\, prime\,\,and\,\,
there\,\,is\,\, a \,\, prime\,\,} q\not=p\,\, {\rm\,\, such\,\, that\,\,}\\
&p^eq^f {\rm\,\,is\,\, a\,\,} WCN \,\,{\rm or\,\,}  
q^ep^f {\rm\,\,is\,\, a\,\,} WCN\}.
\end{split}\end{equation*}
  \begin{conjecture}
For arbitrary given pair $(e,f)$ with $1\le e\le f$ 
and $e+f\ge 3$ the set ${\mathcal Q}_w(e,f)$ has a density 
$1$ with respect to the set of all primes.
 \end{conjecture}
Finally, for every pair $(e,f)$ with $1\le e\le f$ 
and $e+f\ge 3$, and any odd prime $q$  let 
${\mathcal Q}_w(q;e,f)$ be a set defined as
   $$
{\mathcal Q}_w(q;e,f)=\{p:\,p\,\, {\rm is\,\, a\,\, prime\,\,such\,\,that}\,\,
p^eq^f {\rm\,\,is\,\, a\,\,} WCN \,\,{\rm or\,\,}  
q^ep^f {\rm\,\,is\,\, a\,\,} WCN\}.
  $$
 \begin{conjecture}
The union 
  $$
\bigcup_{1\le e\le f<\infty}{\mathcal Q}_w(q;e,f)
  $$ 
is an infinite set.
 \end{conjecture}
Using arguments from Examples 2.32, we immediiately obtain the following 
 result and its corollary.
   \begin{proposition}
Let $p$ and $q$ be two odd distinct primes such that $p<q$ and 
$q-1$ is not divisible by $p$. Let $u$ and $v$ be the smallest 
positive integers for which $p^u\equiv 1(\bmod{\,q-1})$ and 
$q^v\equiv 1(\bmod{\,p-1})$. Then  $p^aq^b$ is a 
$WCN$ if and only if $a$ and $b$ are positive integers 
such that $u\mid a$ and $v\mid b$.
 \end{proposition}
 \begin{corollary}
If $n=p^eq^f$ is a weak Carmichael number, then 
$n=p^{de}q^{lf}$ is also a weak Carmichael number for 
all positive integers $d$ and $l$.
 \end{corollary}
 \begin{corollary}
Let $p$ and $q$ be odd  primes such that $p<q$ and 
$q-1$ is not divisible by $p$. Then 
$p^{e\varphi(q-1)}q^{f\varphi(p-1)}$ is a weak Carmichael number
for arbitrary pair of positive integers $e$ and $f$. 
  \end{corollary}

\begin{remark}
For any $e\ge 2$ and a fixed prime 
$p\ge 3$, consider the set ${\mathcal C}_w(p;e,1)$ of 
{\it WCN} of the form $p^eq$.  Then (cf. Example 2.32)
$n=p^eq$ is in {\it WCN} for some odd prime $q\not= p$ 
if and only if $p-1\mid q-1$ and $q-1\mid p^e-1$, or equivalently,  
$q=(p-1)s+1$ for some divisor $s\ge 2$ of 
$(p^e-1)/(p-1)=p^{e-1}+p^{e-2}+\cdots +1$. 
If $e$ is  even, then assuming $s=2$, that is, $q=2p-1$,
it follows that  $2p-1$ belongs to  ${\mathcal C}_w(p;e,1)$ 
if  and only if  $2p-1$ is  a prime. By using a ``usual" heuristic argument 
based on the {\it Prime number theorem} 
that the probability that an odd integer $m$ is a prime is $2/\log m$,
it follows that the ``expected number" of the elements 
of the set ${\mathcal C}_w(e,1)$ with even $e$ is at least
  $$
2\sum_{p\rm{\,\,odd\,\, prime}}\frac{1}{\log(2p-1)}\ge 
2\sum_{p\rm{\,\,odd\,\, prime}}\frac{1}{2(p-1)} =\infty.
  $$
(Here it is used the well known fact that the sum of reciprocals of
primes diverges).

The situation is somewhat  complicated when $e\ge 3$ is odd. Then 
consider the set of all odd primes $p$ such that $p\equiv 1(\bmod{\,e})$.
Then $\sum_{i=0}^{e-1}{p^i}\equiv 0(\bmod{\,e})$, and take
$q=e(p-1)+1$.  Then  the probability that  
 $q=(p-1)e+1$ is a prime is $e/(\varphi(e)\log (e(p-1)+1)$. Using 
this and the well known fact that 
the series $\sum_{p{\rm\,\,odd\,\, prime} \atop p\equiv 1(\bmod{\,e})}1/p$
diverges, we find that  the ``expected number" of the elements 
that belong to the set  ${\mathcal C}_w(e,1)$ with odd $e\ge 3$ is 
  $$
\frac{e}{\varphi(e)}\sum_{p{\rm\,\,odd\,\, prime}\atop p\equiv 1(\bmod{\,e})}
\frac{1}{\log (e(p-1)+1)}<
\frac{e}{\varphi(e)}\sum_{p{\rm\,\,odd\,\, prime}\atop p\equiv 1(\bmod{\,e})}
\frac{1}{e(p-1)}=\infty.
   $$
The above considerations suggest the conjecture that
${\mathcal C}_w(e,1)$ is infinite set for all $e\ge 2$. 
This conjecture by Corollary 2.37 implies the same conjecture for all sets 
 ${\mathcal C}_w(e,l)$ with $e\ge 2$ and $l\ge 2$.
In accordance to this, some additional 
computations  and the conjecture that 
for any given integer $s\ge 3$, there are infinitely many 
$CN$ with exactly $s$ prime factors 
(cf. a stronger Conjecture 1 in \cite{gp} which asserts that this number 
up to $x$ is at least $x^{1/s+o_s(1)}$), we give the following 
generalized conjecture.
  \end{remark}
  \begin{conjecture}
Let $s\ge 2$ be an arbitrary integer, and let 
$(e_1,e_2,\ldots,e_s)$ be any fixed $s$-tuple  of  integers 
$e_1,e_2,\ldots,e_s$ with $e_1\ge e_2\ge\cdots \ge e_s\ge 1$
and $\sum_{i=1}^se_i\ge 3$.
Then there are infinitely many weak Carmichael numbers $n$ 
with a prime factorization 
$n=p_1^{e_1}p_2^{e_2}\cdots p_s^{e_s}$, where $p_1,p_2,\ldots,p_s$ 
are  distinct odd primes. 
  \end{conjecture}
\begin{remark}
A heuristic argument suggests that for a large odd positive integer 
$n$ which is neither a $CN$ nor a prime power, 
the ``probability" that 
$\sum_{i=1}^{\varphi(n)}r_i^{n-1}\equiv \varphi(n)(\bmod{\,n})$ is equal to 
$(n-1)/2$. Consequently, the number of $WCN$ in the 
interval $[1,n]$ is asymptotically equal to the double harmonic sum
  $2\sum_{k=1}^{[n/2]}1/k$ which is $\sim 2\log n$ as $n$ $\to\infty$.
Furthermore, as noticed above, 
the number of $CN$ in the interval $[1,n]$ is greater than 
$n^{2/7}$ for sufficiently large $n$. 
Moreover, under certain (widely-believed) assumptions about the distribution
of primes in arithmetic progressions, it is shown in \cite[Theorem]{agp}
(see also \cite{gp}) that there are $n^{1-o(1)}$ Carmichael numbers up to
$n$, as had been conjectured in 1956 by Erd\H{o}s \cite{er1} (see also
\cite{sh}). On the other hand, it is known that
the number of prime powers with exponents $\ge 2$ (the sequence A025475 in 
\cite{sl}) up to $x$ (see e.g., \cite[p. 27]{hr}) is given by $O(x^{1/2}\log x)$ 
(more precisely, this number is $2x^{1/2}\log x$).
These considerations suggest the following conjecture.
 \end{remark}
 \begin{conjecture}
The numbers of Carmichael numbers and weak Carmichael numbers 
in the interval $[1,n]$ are  asymptotically equal as $n\to\infty$.
  \end{conjecture}
From Table 1 we see that 2465 and 2821 are (the first) 
twin Carmichael numbers, and  62745 and 63973 are also 
twin Carmichael numbers in the sense of the following definition.
 \begin{definition}     
Two Carmichael numbers are said to be {\it twin Carmichael numbers}
if there is none weak Carmichael number between them.
  \end{definition}
Accordingly to the Conjecture 2.42, we can propose the following 
``twin Carmichael numbers conjecture" which is an immediate consequence 
of  Conjecture 2.42.
 \begin{conjecture}
There are infinitely many pairs of twin Carmichael numbers.
 \end{conjecture}  

 \begin{remark}
 We see from Table 2 that the pairs $(656601, 658801)$ and  

\noindent $(658801, 670033)$ are consecutive twin Carmichael numbers.
  \end{remark}

 \begin{remark}
As noticed in Subsection 1.1,  Lehmer condition implies that
a composite positive integer must be square-free.
This also concerns to   the  Giuga's condition defined 
by the congruence (1.3). Moreover, all $CN$ 
and $k$-Lehmer numbers are  square-free (see Remark 1.2). 
However, from Table 2 we see that  there
are numerous non-square-free composite $WCN$.
\end{remark}
 \begin{remark}
We believe that  the investigation of 
$WCN$ and their distribution would be more 
complicated than those 
on $CN$, Lehmer numbers and Giuga numbers. 
This also concerns to the $k$-Lehmer numbers presented in Remark 1.2
as well as to the Giuga's-like numbers recently investigated in 
\cite{glm} and \cite{lps}.
  \end{remark}
We see from Table 2 that the smallest $WCN$
with three prime distinct factors is  the $CN$  561, 
and  the smallest {\it WCN} with four prime distinct factors is
the $CN$ 41041. Accordingly, 
we propose the following curious conjecture.
 \begin{conjecture}
Let $k$ be an arbitrary integer $\ge 3$. Then the smallest $WCN$
  with $k$ prime distinct factors is  a $CN$.
\end{conjecture}
 \begin{remark}
Under the assumption of Conjecture 2.48, every  $WCN$ less than 
the smallest $CN$ with six distinct prime factors
$3211197185=5\times 19\times 23\times 29\times 37\times 137$
has at most five distinct factors.
 \end{remark}
\subsection{A compuational search of weak Carmichael numbers via the function 
{\tt Carmichael Lambda}}

As noticed above, Carmichael lambda function 
$\lambda(n)$ denotes the size of the largest cyclic subgroup of 
the group $(\Bbb Z/n\Bbb Z)^{*}$ of all reduced residues modulo $n$. 
In other words, $\lambda(n)$ is the smallest 
positive integer $m$ such that $a^m\equiv 1(\bmod{\,n})$ for all
 $a$ coprime to $n$ (Sloane's sequence A002322 \cite{sl}).
This function was implemented in {\tt Mathematica 8}
as the function ``{\tt Carmichael Lambda}". For a fast 
computation  of $WCN$ we can use this function in view of  
the following fact which is immediate  from Theorem 2.4 
and the fact that  for every odd integer
$n=p_1^{e_1}p_2^{e_2}\cdots p_s^{e_s}$, we have
$\lambda(n)={\rm lcm}(\lambda(p_1^{e_1}),\lambda(p_2^{e_2}),\ldots,
\lambda(p_s^{e_s}))$ with $\lambda(p_i^{e_i})=\varphi(p_i^{e_i})=
p_i^{e_i-1}(p_i-1)$ for all $i=1,2,\ldots,s$.
 \begin{proposition}
Let  $n=p_1^{e_1}p_2^{e_2}\cdots p_s^{e_s}$ be an odd composite  
integer, and let $n'=p_1p_2\cdots p_s$. Then $n$ is a weak Carmichael number 
if and only if
   \begin{equation}
\lambda(n')\mid n-1.
   \end{equation}
  \end{proposition}
Proposition 2.50 suggests the following definition introduced by 
Erd\H{o}s in 1948 \cite{er2}.
 \begin{definition}
A positive integer $n>1$ such that $\gcd(n,\varphi(n))=1$
is called a  {\it  $K$-number}.
  \end{definition} 
Erd\H{o}s noticed that $n>1$ is a $K$-number if and only if 
$n$ is a square-free and it is  divisible  by none of the products $pq$
of two distinct  primes $p$ and $q$ with $q\equiv 1(\bmod{\,p})$.
Moreover,  Erd\H{o}s \cite[Theorem]{er2}  
proved that the number of $K$-numbers less than $x$ is 
$\sim xe^{-\gamma}/(\log\log\log x)$, where  $\gamma$ is the Euler's 
constant.  Proposition 2.50 together with Euler totient theorem and 
the definition of Carmichael lambda function easily gives the following
result. 
  \begin{proposition}
Let $n>2$ be a  $K$-number. Then $n$ is odd and
 $n^{d\lambda(n)}$ is 
a weak Carmichael number   for each positive integer $d$. In particular, 
for such a $n$,  $n^{d\varphi(n)}$ is a weak Carmichael number for each 
positive integer $d$.

Furthermore, if $n=p_1p_2\cdots p_s$, then 
  $$
\lambda(n)={\rm lcm}(p_1-1,p_2-1,\ldots,p_s-1).
  $$
  \end{proposition}
Clearly, every Carmichael number is a $K$-number, and hence Proposition 2.52
immediately yields the following result.
  \begin{corollary}
Let $n$ be a Carmichael number. Then both 
$n^{d\lambda(n)}$ and $n^{d\varphi(n)}$ are weak Carmichael numbers
for arbitrary positive integer $d$.
\end{corollary}
  \begin{remark}
For a  fast search  of some special ``types" of $WCN$
it can be used the condition (2.9) to make suitable codes for  these 
purposes. For any fixed $k\ge 2$, let ${\mathcal W_k}$ denote the set of all 
$WCN$ whose prime factorizations contain exactly $k$ primes. 
Further, for positive integers $a,b,c,d$  with $a<b$ and $c<d$ take
${\mathcal W}_k(a,b)={\mathcal W_k}\bigcap [a,b]$, 
and let ${\mathcal W}_k(a,b;c,d)$ be the set of all 
elements   in ${\mathcal W}_k(a,b)$ whose greatest  prime divisor belongs 
to the interval  $[c,d]$. Clearly, ${\mathcal W}_k(a,b;c,b)$ 
is a set of all elements   in ${\mathcal W}_k(a,b)$ whose greatest  prime 
divisor is grater than equal to $c$.
The cardinilities  of sets 
${\mathcal W}_k$, ${\mathcal W}_k(a,b)$ and 
${\mathcal W}_k(a,b;c,d)$ are denoted by $W_k$, $W_k(a,b)$ and 
$W_k(a,b;c,d)$, respectively. For such a set ${\mathcal W}_k(a,b;c,d)$
define 
  $$
{\mathcal P}_k(a,b;c,d)=\{p:\,p{\rm\,\,is\,\,a \,\, prime\,\, with\,\,}
p\mid n\,\,{\rm for\,\, some\,\,} n\in  
{\mathcal W}_k(a,b;c,d)\},
  $$
and let $p_k(a,b;c,d)$ be a prime defined as 
    $$
p_k(a,b;c,d)=\max\{p:\,p\in {\mathcal W}_k(a,b;c,d)\},
    $$
and let $w_k(a,b;c,d)$ be the smallest number in ${\mathcal W}_k(a,b;c,d)$ 
which is divisible by $p_k(a,b;c,d)$.

If ${\mathcal C_k}$ denotes the set of all 
Carmichael numbers whose prime factorizations contain exactly $k$ primes,
then in the same manner as above, we define the sets  
${\mathcal C}_k(a,b)$, ${\mathcal C}_k(a,b;c,d)$ 
and related numbers $C_k(a,b)$ and ${\mathcal C}_k(a,b;c,d)$  
associated to ${\mathcal C_k}$.  Also let $C(a,b)$ 
be the number of all  Carmichael numbers that belong to the interval 
$[a,b]$, and let $P(a,b)$ be the number of all odd  prime powers that belong to 
the interval $[a,b]$. To save the space, the set ${\mathcal C}_k(b)$
will be  denoted by  ${\mathcal C}_k(1,b)$ and its cardinality 
by $C_k(b)$. Also let  $C(b)=\sum_{k\ge 3}C_k(b)$
be the number of Carmichael numbers which are less than $b$.
 Similarly, we define the set  ${\mathcal W}_k'(b)$ 
consisting of all $WCN$ in ${\mathcal W}_k(1,b)$ which are 
not $CN$. The cardinality of ${\mathcal W}_k'(b)$ is denoted here 
as $W_k'(b)$. Denote by  $W'(b)$ the number of all $WCN$ up to $b$
which are neither $CN$ nor prime powers; that is,
$W'(b)=\sum_{k\ge 2}W_k'(b)$. 
   \end{remark}

Here we present  a computational search of $WCN$ that belong to 
${\mathcal W}_2$, i.e., of integers $n=p^eq^f$ with primes $3\le p<q$ and 
some positive integers $e$ and $f$. In particular, our code 
in {\tt Mathematica 8} for determining different sets of the form ${\mathcal W}_2(a,b;c,d)$ 
gives results presented in Table 3
(recall that all non-prime powers  weak Carmichael numbers less than 
$2\times  10^6$ are presented in Table 2). \\

\begin{center}
{\bf Table 3.} Numbers  
$W_2(a,b;c,d)$, $p_2(a,b;c,d)$, $w_2(a,b;c,d)$ (all written in Table 2 
without ``$(a,b;c,d)$"), $P(a,b)$ and $C(a,b)$
for $n<10^{12}$ 
{\small
\begin{tabular}{|c|c|c|c|c|c|}\hline
$(a,b;c,d)$  &  $W_2$  & $p_2$ & $w_2$ & $P(a,b)$& $C(a,b)$\\\hline
$(1,10^6;1,10^6)$  &  $107$  & 463 & $856087=43^2\cdot 463$ &218 & 43\\
$(10^6,2\cdot 10^6;1,2\cdot 10^6)$  & 25   & 733 & $1610401=13^3\cdot 733$ 
&65& 12\\\hline
total  &  132  &  & &283 & 55\\\hline
$(2\cdot 10^6,10^7;1,10^3)$  &  69  & $937$ & $2632033=53^2\cdot 937$
&250  &50 \\
$(2\cdot 10^6,10^7;10^3,10^4)$  &  5  & $1861$ & $6924781=61^2\cdot 1861$
&-&-\\
$(2\cdot 10^6,10^7;10^4,10^7)$  &  $0$  & $-$ & -&-&-\\\hline
total  &  74  &  & &250 & 50\\\hline
$(10^7,10^8;1,10^3)$  &  120  & 997 & $27805333=167^2\cdot 997$
& 846 &150\\
$(10^7,10^8;10^3,10^4)$  & 43  &  & $81390625=5^6\cdot 5209$
-&-&-\\
$(10^7,10^8;10^4,10^8)$  & 0 &  & 
-&-&-\\\hline
total  &  163  &  & & 846 &150\\\hline
$(10^8,10^9;1,10^3)$  & 156    &  & & 2282 &  391\\
$(10^8,10^9;10^3,10^4)$  & 109    &  & & &  \\
$(10^8,10^9;10^4,10^9)$  &  9    &  & & &  \\\hline
total  &  274  &  & & 2282& 391\\\hline
$(10^9,10^{10};1,10^3)$  & 211 &  & & 6391 &  901\\
$(10^9,10^{10};10^3,10^4)$  & 112    &  & & &  \\
$(10^9,10^{10};10^4,10^{10})$  & 74   &  & & &  \\\hline
total  & 397   &  & &6391 & 901\\\hline
$(10^{10},10^{11};1,10^3)$  & 247    &  & & 18069 & 2058  \\
$(10^{10},10^{11};10^3,10^4)$  & 93   &  & & &   \\
$(10^{10},10^{11};10^4,10^{11})$  & 253    &  & & &  \\\hline
total  & 593     &  & &18069 & 2058 \\\hline
$(10^{11},10^{12};1,10^3)$  & 266     &  & & 51911 & 4636  \\
$(10^{11},10^{12};10^3,10^4)$  & 220   &  & & &   \\
$(10^{11},10^{12};10^4,10^{12})$  & 689    &  & & &  \\\hline
total  & 1175    &  & & 51911& 4636 \\\hline\hline
 total up to $10^{12}$  & 2808   &  & & 80032  &    8241 \\\hline
   \end{tabular}}
 \end{center}
\vspace{2mm}

Let $W_3'(N)$ be the number of all 
$n=p^aq^br^c\in {\mathcal WCN}\setminus {\mathcal CN}$ up to $N$
with odd primes $p<q<r$.
Using  this notation and the  previous notations, counting related numbers in 
Table, we arrived to the following table.\\  

\vfill\eject

\begin{center}
{\bf Table 4.} Numbers  
$C_k(N)$ and   $W_k'(N)$ with $k=2,3,4,5$ and 

\noindent $N\in\{10^3,10^4,10^5,10^6,2\cdot 10^6\}$. 
{\small
\begin{tabular}{|c|c|c|c|c|}\hline
Pairs $(N,k)$  & $C_k(N)$  & $C(N)$ &    $W_k'(N)$ & $W'(N)$\\\hline
$(10^3,2)$ & - & 1  &6  &6  \\  
$(10^4,2)$ & - & 7 & 22 & 25 \\  
$(10^5,2)$ & - & 16 & 51  & 70  \\  
$(10^6,2)$ & -  & 43   & 107  & 192 \\  
$(2\cdot 10^6,2)$ & - & 55 & 132  & 243  \\\hline
$(10^3,3)$ & 1  &   & 0  &  \\  
$(10^4,3)$ & 7  &   &     3 &  \\  
$(10^5,3)$ & 12   &   &    18   & \\  
$(10^6,3)$ & 23  &   &    68  &    \\  
$(2\cdot 10^6,3)$ & 30  & & 89  &   \\\hline
$(10^4,4)$ & 0  &  & 0  & \\  
$(10^5,4)$ & 4  &  & 1  & \\  
$(10^6,4)$ & 19  &  & 17  & \\  
$(2\cdot 10^6,4)$ & 23 & &  22 & \\\hline 
$(10^5,4)$ & 0  &  & 0  & \\  
$(10^6,5)$ & 1  &  & 0  & \\  
$(2\cdot 10^6,5)$ & 2 & &  0 & \\\hline\hline 
total up to $N= 2\cdot 10^6$ & 55 & 55 & 243  & 243 \\ \hline 
   \end{tabular}}
  \end{center}
\vspace{2mm}

For a search of $W_3'(N)$ with $N\le 10^{12}$, we use 
a characetrization of $WCN$  given by Theorem 2.4. The three-component  
Carmichael number counts, $C_3(N)$, presented in the second column of Table 5,
are taken from the Granville and Pomerance paper \cite{gp}. These counts were
calculated by R. Pinch, J. Chick, G. Davies and M. Williams 
(cf. \cite[Table 2]{du}). 

\vfill\eject

   \vspace{2mm}

\begin{center}
{\bf Table 5.} Numbers  
$C_3(N)$ and   $W_3'(N)$ with $N\in\{10^3,10^4,10^5,10^6,2\cdot 10^6,10^7,
10^8,\ldots,10^{12}\}$. 

{\small
\begin{tabular}{|c|c|c|c|c|}\hline
$N$  & $C_3(N)$  & $W_3'(N)$ & $n=p^aq^br^c\in {\mathcal W}_3'(N)$ 
& $n=pqr\in {\mathcal C}_3'(N)$ \\
&  &  &  with 
a maximal $r$&  with a maximal $r$ \\\hline
$10^3$ & 1 & 0  & - & $561=3\cdot 11\cdot 17$\\
$10^4$ & 7 & 3 & $6525=3^2\cdot 5^2\cdot 29$ & $8911=7\cdot 19\cdot 67$\\
$10^5$ & 12 & 18 & $25425=3^2\cdot 5^2\cdot 113$ & $52633=7\cdot 73\cdot 103$\\
$10^6$ & 23  & 68 & $750925=5^2\cdot 7^2\cdot 613$& 
$530881=13\cdot 97\cdot 421$\\
$2\cdot 10^6$ & 30& 89   &$1269621=3^3\cdot 59\cdot 797$ &
$1193221=31\cdot 61\cdot 631$\\
$10^7$ & 47 & 186  & $8927425=5^2\cdot 13^2\cdot 2113$ & $8134561=37\cdot 109\cdot 2017$\\
$10^8$ & 84 & 413 & $52280425=5^2\cdot 409\cdot 5113$ 
& $67902031=43\cdot 271\cdot 5827$\\
$10^9$ & 172& 863 & $954036721=11^2\cdot 19^2\cdot 21841$ & 
$962442001=73\cdot 601\cdot 21937$\\
$10^{10}$ &335 & 1590 & $4465266751=11^3\cdot 71\cdot 47251$& 
$8863329511=211\cdot 631\cdot 66571$\\
$10^{11}$ &590 & 2866 & $79183494081=3^4\cdot 17^3\cdot 198977$ & 
$74190097801=151\cdot 2551\cdot 192601$\\
$10^{12}$ &1000 &4291  & $800903953125=3^4\cdot 5^6\cdot 632813$ &
$921323712961=673\cdot 2017\cdot 678721$\\
$10^{13}$ &1858 & & &\\
$10^{14}$ &3284 & & &\\
$10^{15}$ &6083 & & &\\
$10^{16}$ &10816 & & &\\
$10^{17}$ &19539 & & &\\
$10^{18}$ &35586 & & &\\
$10^{19}$ &65309 & & &\\
$10^{20}$ &120625 & & &\\\hline
 \end{tabular}}
\begin{remark}
Notice that prime factors of every $WCN$ $n=p^aq^br^c$ in the fourth column 
of Table 5 besides  the number $6525$ satisfy the equlity $r-1=(p^aq^b-1)/2$.
Similarly, the prime factors $p,q,r$ of $CN$ $561,8911$, $8134561, $67902031, 
$962442001$, $8863329511$, $74190097801$ and $921323712961$
in the last column of Table 5 satisfy the equality  $r-1=(pq-1)/2$.
 \end{remark}
 \end{center}
\begin{remark}
Let $\alpha=\alpha(N)$ denote the real number such that $C_3(N)=N^{\alpha}$,
and let $\beta=\beta(N)$ be the real number such that $W_3'(N)=N^{\beta}$.
Then from data in Table 5 we find that $\alpha(10^6)=0.227$,
 $\alpha(10^9)=0.248$, 
$\alpha(10^{12})=0.250$, $\alpha(10^{15})=0.252$  
$\alpha(10^{18})=0.253$, $\alpha(10^{21})=0.255$,
 $\beta(10^6)=0.305$  $\beta(10^9)=0.326$ and  $\beta(10^{12})=0.303$.
 \end{remark}

  \subsection{$k$-Lehmer numbers and weak Carmichael numbers}
Quite recently, J.M. Grau and A.M. Oller-Marc\'en 
 \cite[Definition 1]{gm2} weakened Lehmer property by introducing
the concept of $k$-Lehmer numbers.  For given  positive integer $k$, a  
$k$-{\it Lehmer number} is a composite integer $n$ such that 
$\varphi(n)\mid (n-1)^k$. Hence, if we denote by $L_k$ the set 
  $$
L_k:=\{n\in\Bbb N:\, \varphi(n)\mid (n-1)^k\},
  $$
then $k$-Lehmer numbers are the composite elements of $L_k$.
Clearly, $L_k\subseteq L_{k+1}$ for each $k\in \Bbb N$, and define
  $$
L_{\infty}:=\bigcup_{k=1}^{\infty}L_k.
  $$
Then it can be easily shown  that (see \cite[Proposition 3]{gm2})
   $$
L_{\infty}:=\{n\in\Bbb N:\,{\rm rad}(\varphi(n))\mid n-1\}.
  $$
This immediately  shows that 
if $n$ is a Carmichael number, then $n$ also belongs to the set 
$L_{\infty}$ (\cite[Proposition 6]{gm2}). This leads to the following 
characterization of Carmichael numbers which slightly modifies Korselt's 
criterion.
   \begin{proposition} 
A composite number $n$ is a Carmichael number if and only if 
${\rm rad}(\varphi(n))\mid n-1$ and $p-1\mid n-1$ for every prime 
divisor $p$ of $n$.
 \end{proposition}
 Obviously, the composite elements of $L_1$ are precisely the Lehmer numbers
and the Lehmer property asks whether $L_1$ contains composite numbers
or not. Nevertheless, for all  $k>1$, $L_k$ always contains composite 
elements (cf. Sloane's sequence A173703 in OEIS \cite{sl} 
which presents $L_2$).  
For further radically weakening the Lehmer and Carmichael conditions 
see \cite{mcn}.

 As an immediate consequence of Proposition 2.57 and Theorem 2.4
we obtain the following characterization of Carmichael numbers.
   \begin{corollary} 
A composite number $n$ is a Carmichael number if and only if 
$n$ is a weak Carmichael number and
${\rm rad}(\varphi(n))\mid n-1$.
 \end{corollary}
    \subsection{Super Carmichael numbers}
The fact that there are  infinitely many  weak Carmichael numbers
suggests  the following definition.

\begin{definition}
A weak Carmichael number  $n$ is said to be a 
{\it super  Carmichael number} if 
    \begin{equation}
 \sum_{\gcd(k,n)=1\atop 1\le k\le n-1}k^{n-1}\equiv \varphi(n)\pmod{n^2},
     \end{equation}
where the summation ranges over all $k$ such that $1\le k\le n-1$ 
and $\gcd(k,n)=1$.
 \end{definition}

The following characterization of super Carmichael numbers may be useful  
for computational purposes.

     \begin{proposition}
An odd composite positive integer $n>1$ is a super Carmi-\
chael number if and 
only if
  \begin{equation}
2\sum_{i=1}^{\varphi(n)/2}r_i^{n-1}+n\sum_{i=1}^{\varphi(n)/2}r_i^{n-2}
\equiv \varphi(n)\pmod{n^2}
  \end{equation}
where $r_1<r_2<\cdots <r_{\varphi(n)}$ are all reduced 
residues modulo $n$.
  \end{proposition}
Here, as always in the sequel, the super Carmichael number(s) will be often 
denoted by $SCN$.
  \begin{remark} 
Using Proposition 2.18, some computations and 
a heuristic argument, we can assume that the probability that a prime power
$p^s$ with an odd prime $p$ and $s\ge 2$, is a $SCN$ 
is equal to $1/p^s$. Furthermore, applying (2.11), a computation 
in {\tt Mathematica 8} shows that none prime power $p^s$ less than 
$3^{16}$ with $s\ge 2$ and $p\le p_{847}=6553$  
is   a $SCN$. This together 
with the identity $\sum_{i=s+1}^{\infty}{1/p^i}=1/(p^s(p-1))$ 
  \begin{equation}
\Sigma_1:=\sum_{p\,\,{\rm odd\,\, prime}\atop s\ge 2 \,\, {\rm and}
\,\,p^s\ge 3^{16}}\frac{1}{p^s}=8.91952\cdot 10^{-7}.
  \end{equation}
On the other hand,  a computation also gives 
    \begin{equation}\begin{split}
\Sigma_2: &=\sum_{p\,\,{\rm \,\, prime} \atop p>6553}\sum_{i=2}^{\infty}
\frac{1}{p^i}=\sum_{p\,\,{\rm \,\, prime}\atop p>6553}\frac{1}{p(p-1)}
<0.00016< \sum_{k=6552}^{\infty}\frac{1}{k^2}\\
&=\zeta(2)-\sum_{k=1}^{6551}\frac{1}{k^2}=0.000152381.
   \end{split}\end{equation}
Using (2.12), (2.13) and the  fact that
none prime power $p^s$ less than $3^{16}$ with 
$s\ge 2$ and $p\le p_{847}=6553$  is  a $SCN$,
we find that the expected number of $SCN$ that occur  in the set of all prime 
powers of the form $p^s$ with $s\ge 2$ is 
  $$
\Sigma_1+\Sigma_2<0.001525.
  $$
Using the above estimate, we can propose the following conjecture.
 \end{remark}
     \begin{conjecture}
Let $p$ be any odd prime. Then none prime power $p^f$ with $f\ge 2$ is 
a super Carmichael number. 
     \end{conjecture}
Notice that  by using a  result of  I.Sh. Slavutskii \cite{sl1}, 
it is proved in Section 4 the following result.
     \begin{proposition}
Let $p$ be an odd prime greater than $3$. 
Then a prime power $p^f$ with $f\ge 2$ is 
a super Carmichael number if and only if the numerator of the Bernoulli 
number $B_{(p^{2f}-p^{2f-1}-1)(p^{2f}-p^{2f-1}-p^f+1)}$ is divisible by 
$p^{f+1}$.
       \end{proposition}
 \begin{remark}
Using Table 2, a computation in {\tt Mathematica} 8 shows that there are none 
$SCN$ less than $2\times 10^6$. 
Notice that by using  Harman's result \cite{har} given in Subsection 1.1, 
it follows that the ``probability" that  a sufficiently large positive integer
$n$ is a $CN$ is greater than $n^{0.33}/n=1/n^{0.67}$. 
Using this, 
some ``little" computations  and a heuristic argument,
we can assume that the  ``probability" that  a large number $n$ is 
a $SCN$ is greater than 
$1/(n\cdot n^{0.67})=1/n^{1.67}$. It follows that the 
``expected number" of   $CN$ in a large 
 interval $[1,N]$  is greater than
  $$
\sum_{n=1}^N\frac{1}{n^{1.67}}
  $$
which tends to $\zeta(1.67)=2.11628$ as $N\to\infty$.
However, as noticed above, under certain 
assumptions about the distribution of primes in arithmetic progressions, 
it is shown in \cite[Theorem]{agp}
that there are $x^{1-o(1)}$ Carmichael numbers up to $x$.
For this subject, see also \cite{bp}. 
 It was also given in  \cite{pom1}  a heuristic argument that
this number is $x^{1-\varepsilon(x)}$, where 
$\varepsilon(x)=(1+o(1))\log\log\log x/(\log\log x)$. 
This argument is supported by counts of $CN$ mostly done  
in 1975 by J.D. Swift \cite{sw}, in 1990 by G. Jaeschke \cite{ja},  
by R. Pinch \cite{pi1} in 1993 and R. Pinch \cite{pi2} in 2006. 
Accordingly, using the previous arguments, we can assume 
that the ``probability" that  a large number $n$ is 
a $SCN$ is about $1/n^{o(1)}$. It follows that the 
``expected number" of  Carmichael numbers in a very large 
 interval $[1,N]$  is greater than
  $$
\sum_{k=1}^N\frac{1}{k}.
  $$ 
This together with the fact that  $\sum_{k=1}^{\infty}1/k=+\infty$ 
motivates the following conjecture.
  \end{remark}

  \begin{conjecture}
There are infinitely many super Carmichael numbers.
  \end{conjecture}
\begin{remark} 
A heuristic argument and considerations given in Remark 2.64 
suggest that a search for $SCN$
would be have a ``chance" only between $SCN$. 
In other words, $SCN$ ``probably" can occur only between 
$CN$. Hence, we propose  the following conjecture.
 \end{remark}

 \begin{conjecture} 
Every super Carmichael number is necessarily a Carmichael number. 
  \end{conjecture}
  \begin{remark} 
Because of Conjecture 2.67, we have omitted the word  ``weak" in the name  
``super Carmichael number" given in Definition 2.59.
 \end{remark}
 \begin{remark}
In order to examine whether given  $CN$ $n$ is also 
a $SCN$, it is natural to proceed as follows.
Take $n=p_1p_2\cdots p_s$, where $p_1,p_2,\ldots ,p_s$
are distinct odd primes. Then by the congruence (2.11) of Proposition 
2.60, with $r_i$ defined in Proposition 
2.60,   it follows that $n$ is a $SCN$ if and only if
  \begin{equation}
c_n:=2\sum_{i=1}^{\varphi(n)/2}r_i^{n-1}+n\sum_{i=1}^{\varphi(n)/2}r_i^{n-2}
-\varphi(n)\equiv \pmod{n^2}.
  \end{equation}
Clearly, the congruence (2.14) holds if and only if
   \begin{equation}
c_n\equiv 0\pmod{p_i^2}\quad {\rm for\,\, each\,\,} i=1,2,\ldots,s.
   \end{equation}
Now taking  $n-1=q_i\varphi(p_i^2)+l_i=q_ip_i(p_i-1)+l_i$ and 
$n-2=t_i\varphi(p_i)+u_i=t_i(p_i-1)+u_i$ with integers 
$q_i,t_i\ge 1$ and $0\le l_i,u_i\le p_i-1$ for all $i\in\{1,2,\ldots,s\}$.
Then by Euler totient theorem and (2.15) and the fact 
that $\varphi(p_i^2)=p_i^2-p_i$, the congruence (2.14) is satisfied
if and only if 
  \begin{equation}
c_n(p_i):=2\sum_{i=1}^{\varphi(n)/2}r_i^{l_i}+
n\sum_{i=1}^{\varphi(n)/2}r_i^{u_i}p_i\equiv 0\pmod{p_i^2}\,\,
{\rm for\,\, each}\,\, i=1,2,\ldots,s.
  \end{equation}
Taking $m_i=\max\{l_i,u_i\}$ for all $i=1,2,\ldots, s$, without 
loss of generality we can suppose that $m_1\le m_2\le\cdots \le m_s$.
Then we firstly verify the congruence (2.16) for $i=1$.
If (2.16) is not satisfied modullo $p_1^2$, then we conclude that
$n$ is not a $SCN$. Otherwise, we continue
the computation passing  to $p_2$ etc. Of course, we finish the computation
to a first index  $i$ for which $c_n(p_i)\not\equiv (\bmod{\,p_i^2})$.
If it is obtained that $c_n(p_i)\equiv (\bmod{\,p_i^2})$
for each $i=1,2,\ldots, k$, then we conclude that $n$ is a $SCN$.
 \end{remark}

 \subsection{Weak Carmichael numbers and the Fermat primality test}

Gauss \cite{ga} (Article 329 of 
{\it Disquisitiones Arithmeticae},  1801), let. 329]) wrote:
\vspace{2mm}

{\it The problem of distinguishing prime numbers from composite numbers is one
of the most fundamental and important in arithmetic. It has remained as a central 
question in our subject from ancient times to this day...}
  \vspace{2mm}

On October 18th, 1640 Fermat wrote, in a letter to his confidante 
Frenicle, that the fact that $n$ divides $2^n-2$ whenever $n$ is prime is not
an isolated phenomenon. Indeed that, if $n$ is prime then $n$ divides $a^n-a$
for all integers $n$; which implies that if $n$ doesn't divide $a^n-a$ for some
integer $a$ then $n$ is composite.

\vspace{2mm}

As noticed above, as ``false primes" 
Carmichael numbers are quite famous among specialists in number theory,
as they are quite rare and very hard to test.
Accordingly, these numbers present a major problem for 
Fermat-like primality tests. Here we give some remarks 
on Carmichael and weak Carmichael numbers closely related to the 
Fermat primality test. 

{\it Fermat little theorem} says that if $p$ is a prime and the integer $a$ is 
not a multiple of $p$, then 
 \begin{equation}
a^{p-1}\equiv 1\pmod{p}.
 \end{equation} 
If we want to test if  $p$ is prime, then we can pick random $a$'s in 
the interval and see if the congruence holds. If the congruence does not hold 
for a value of $a$ then $p$ is composite. If the congruence does hold for many 
values of $a$, then we can say that $p$ is ``probable prime". It might  
be in our tests that we do not pick any value for a such that the congruence 
(2.17)  fails. Any $a$ such that $a^{n-1}\equiv 1(\bmod{\, n})$
when $n$ is composite is called a {\it Fermat liar}. In this case $n$ is 
called {\it Fermat pseudoprime to base} $a$. If we do pick an integer  $a$ 
such that $a^{n-1}\not\equiv 1(\bmod{\, n})$, then $a$ is called 
a {\it Fermat witness} for the compositeness of $n$.
Clearly, a Carmichael number $n$ is a composite integer that is
 {\it Fermat-pseudoprime} to base $a$ for every $a$ with $\gcd (a,n)=1$. 
On the other hand, it is known that  for ``many" (necessarily even) integers 
$n$ the  congruence $a^{n-1}\equiv 1(\bmod{\,n})$
 is satisfied only when $a\equiv 1(\bmod{\,n})$ 
(this is Sloane's sequence A111305 of ``unCarmichael numbers" \cite{sl}; cf. 
Sloane's A039772 \cite{sl}).  
For any integer $n>1$ let ${\mathcal F}(n)$ be the set defined as 
  $$
{\mathcal F}(n)=\{a\in \Bbb Z/n\Bbb Z:\,a^{n-1}\equiv 1 \pmod{n}\},
 $$
and let  $F(n)=\# {\mathcal F}(n)$, that is, $F(n)$ 
is a number of  residues $a$ modulo $n$
such that $a^{n-1}\equiv 1(\bmod{\,n})$ 
($F(n)$ is Sloane's sequence A063994). Therefore,  
   $$
F(n)=\# \{a\in \Bbb Z/n\Bbb Z:\, a^{n-1}\equiv 1(\bmod{\, n})\},
  $$
that is,  $F(n)$ is a number of Fermat liars for  $n$.
Clearly, ${\mathcal F}(n)$ is a subgroup of the multiplicative group
$\left(\Bbb Z/n\Bbb Z\right)^{*}$. If $n=p$ is a prime, then $F(p)=p-1$ and 
${\mathcal F}(p)=\left(\Bbb Z/p\Bbb Z\right)^{*}$, i.e., ${\mathcal F}(p)$
is the entire group of reduced residues modulo $p$. 

  The following  elegant and simple formula for 
$F(n)$ was established by Monier \cite[Lemma 1]{mo} and Baillie and Wagstaff 
\cite{bwa} (also see \cite{agp2}):
    \begin{equation}
F(n)=\prod_{p\mid n}\gcd(p-1,n-1).
   \end{equation}
We also define the sequence $f(n)$ with $n\ge 2$ as
  \begin{equation}
f(n)=\frac{F(n)}{\varphi(n)}=\prod_{p\mid n}
\frac{\gcd(p-1,n-1)}{(p-1)p^{e_p-1}},
  \end{equation} 
where $n=\prod_{p\mid n}p^{e_p}$.
 \begin{remark}
Recall that the index of every $WCN$ up to $26353$
$n$ presented in Table 2  denotes 
a related value $F(n)$ (for example, $F(26353)=1296$). 
Of course,  $F(n)=n-1$ if and only if $n$ is a prime or a $CN$.  
At the other extreme, there are infinitely many numbers $n$ for which
$F(n)=1$. In particular, (2.18) immediately implies that 
$F(2p)=1$ for every prime $p$. 
It is possible to show (see \cite{ep}) that while these numbers $n$ with
 $F(n)=1$ have asymptotic density 0, they are much more common than primes.
The normal and average size of $F(n)$ for $n$ composite were studied
in 1986 \cite{ep}. By Lagrange theorem, $F(n)\mid \varphi(n)$ for any $n$.
It was proved in \cite[p. 263]{ep} that 
$F(n)=\varphi(n)/k$ for an integer $k$ implies $\lambda(n)\mid k(n-1)$,
where $\lambda(n)$ is the {\it Carmichael lambda function} denoting  a 
smallest positive integer such that $a^{\lambda(n)}\equiv 1(\bmod{\,n})$
for all $a$ with $\gcd(a,n)=1$. Moreover, it was proved in 
\cite[Theorem 6.6]{ep} that if $k$ is odd or $4\mid k$, then there are 
infinitely many $n$ with $F(n)=k$. If $k\equiv 2(\bmod{\,4})$, then 
the equation $F(n)=k$ has infinitely many solutions $n$ or no solutions $n$ 
depending on whether $k=p-1$ for some prime $p$. In particular, the density 
of the range of $F$ is $3/4$. It was also observed in \cite[p. 277]{ep} that the 
{\it universal exponent} $L(n)$ for the group of reduced residues $a$  modulo
$n$ for which $a^{n-1}\equiv 1(\bmod{\,n})$,  is equal to 
${\rm lcm}\{(p-1,n-1):\,p\mid n\}$, and  that $L(n)=\lambda(n)$  
if and only if $F(n)=\varphi(n)$. 
Moreover, $F(n)\mid \varphi(n)$ for all $n\ge 2$.  

Applying  Theorem 2.4 to the formula (2.18),
we immediately get the following result.
 \end{remark}
 \begin{proposition}
A composite positive integer $n$ is a weak Carmichael number
if and only if
  \begin{equation}
F(n)=\prod_{p\mid n}(p-1)
 \end{equation}
where the product is taken over all primes $p$ such that  $p\mid n$.
Furthermore, a composite positive integer $n$ is a Carmichael number  
if and only if $F(n)=\varphi(n)$.
 \end{proposition}
The equality  (2.19)  immediately gives 
  \begin{equation*}
f(n)= \prod_{p\mid n}\frac{\gcd(p-1,n-1)}{(p-1)p^{e_p-1}}\le 
\prod_{p\mid n}\frac{(p-1)}{(p-1)p^{e_p-1}}=
\prod_{p\mid n}\frac{1}{p^{e_p-1}},
  \end{equation*}
whence we have the following result.
 \begin{corollary} 
Let $n>1$ be a positive integer. Then
 \begin{equation*}
f(n)\le \prod_{p\mid n}\frac{1}{p^{e_p-1}},
  \end{equation*}
where equality holds if and only if  $n$ is a weak Carmichael number.
  \end{corollary} 

Of course, it can be of interest to consider the function $f(n)$ 
restricted to the set of positive integers which are not 
$CN$. For this purpose, we will need the following definition.
 \begin{definition}
Let $k\ge 2$ be a positive integer. 
An integer $n=p_1p_2\cdots p_s>1$ with odd primes $p_1,p_2,\ldots ,p_s$ and 
$s\ge 2$, is said to be an {\it almost Carmichael number of order $k$} if  
the following conditions are satisfied:

$\,\,\,(i)$ $p_j-1\mid k(n-1)$ for a fixed $j\in\{1,2,\ldots,s\}$,

$\,(ii)$ $n-1$ is  divisible by $m(p_j-1)$ for none 
$m\in\{1,\ldots,k-1\}$ and 

$(iii)$ $p_i-1\mid n-1$ for all $i\in\{1,2,\ldots,s\}$ such that $i\not= j$.
\end{definition}
\begin{remark} 
A computation shows that  there exist ``numerous" 
almost Carmi-\
chael numbers of order $2$. First notice that  a product $pq$  of  
two distinct odd primes $p$ and $q$ with $p<q$ is an  almost Carmichael number 
of order $2$ if and only $q=2p-1$. Recall that such a $p$  is a {\it Sophie 
Germain-type prime}. Namely, if both $p$ and $2p+1$ are primes, then   
$p$  is called a {\it Sophie Germain prime}, and it was conjectured that
there are infinitely many Sophie Germain primes. Notice that
this conjecture as well as the conjecture that there are infinitely 
many primes $p$ such that $2p-1$ is also a prime, are particular cases 
of a more general {\it Prime-$k$-tuples conjecture}
due to Dickson in 1904 (see e.g., \cite[p. 250]{r2}). 
Furthermore, {\tt Mathematica 8} gives 
numerous ``three-component"  almost Carmichael numbers of order $2$ in the 
set $\{n:\, n=q_iq_jq_k\,\,{\rm and}\,\,2\le i<j<k\le 1000\}$. 
  \end{remark} 

Proposition 2.71, Corollary 2.72 and the formula (2.19)
easily yield the following result.
 \begin{proposition}
The following assertions about a composite positive integer 
$n$ are true.

$\,\,\,(i)$ $f(n)\le 1$, and equality holds if and only if $n$
is a Carmichael number.

$\,(ii)$ If $n$ is not a Carmichael number, then $f(n)\le 1/2$
and equality holds if and only if  $n$  is an almost Carmichael number of 
order $2$.

$(iii)$ If $n$ is neither a Carmichael number nor an 
almost Carmichael number of order $2$, then $f(n)\le 1/3$
and equality holds if and only if  $n$  is  
an almost Carmichael number of order $3$ or
$n$ is a weak Carmichael number with the prime factorization 
 $n=3^2p_2\cdots p_s$, where $3<p_2<\cdots <p_s$ are odd  primes. 
 \end{proposition}
 \begin{remark}
  If $n$ is a composite integer 
which is not $CN$, then by (ii) of Proposition 2.75
we have 
 \begin{equation}
F(n)\le \frac{\varphi(n)}{2}\le \frac{n-1}{2}. 
  \end{equation}
This shows that for such a $n$, at least half  of the integers 
$a$ in the interval $[1,n-1]$ are  
Fermat liars for $n$ (the so-called false witnesses for $n$ \cite{ep}). 
These facts lead to the folowing test: given a positive 
integer $n$, pick $k$ different positive integers less than $n$ and 
perform the Fermat primality test on $n$ for each of these bases; 
if $n$ is composite and it is not a Carmichael number, 
then the probability that $n$ passes all $k$ tests 
is less than $1/2^k$.
 \end{remark}
\begin{remark} 
Let $\mathcal{WCN}_3$ denote 
the set of all $WCN$ of the form $n=3^2p_2\cdots p_s$, 
where $3<p_2<\cdots <p_s$ are odd  primes.
(cf. (iii) of  Proposition 2.75).
It is easy to see $45=3^2\cdot 5$ is the only 
number in the set $\mathcal{WCN}_3$ whose prime factorization 
contains only one prime greater than $3$.  
From Table 2 we see that up to  $2\cdot 10^6$ there are still six numbers
belonging to the set $\mathcal{WCN}_3$;
these numbers are 
$13833=3^2\cdot 29\cdot 53$, 
$203841=3^2\cdot 11\cdot 29\cdot 71$, 
$321201=3^2\cdot 89\cdot 401$, 
$1080801=3^2\cdot 29\cdot 41\cdot 101$ and
$1658385=3^2\cdot 5\cdot 137\cdot 269$.
A computation shows that 
 $$
\{9m: m=pq\,\,{\rm with}\,\, p\,\, {\rm and}\,\,q \,\,{\rm primes}\,\,
\,\,{\rm such \,\, that}\,\, 3<p<q<10^5\}\cap 
\mathcal{WCN}_3=\Phi.
  $$
Furthermore, if $q_k$ denotes the $k$th prime, then  the set 
$\{9m:\, m=q_iq_jq_k\,\,{\rm and}\,\,2\le i<j<k\le 1000\}$
contains the following five numbers of the set $\mathcal{WCN}_3$ 
which  are greater than $2\cdot 10^6$: 
$8074881= 3^2\cdot 17\cdot 89\cdot 593$,
$19678401=3^2\cdot 17\cdot 41\cdot 3137$, 
$95682861=3^2\cdot 29\cdot 53\cdot 6917$,
$359512011=3^2\cdot 23\cdot 467\cdot 3719$ and
$1955610801=3^2\cdot 53\cdot 1433\cdot 2861$.
 \end{remark}

 \begin{remark}
If $n$ is a $CN$, then as noticed above $F(n)=\varphi(n)$. 
Erd\H{o}s and Pomerance  \cite[Section 6]{ep} conjectured 
that not only  are there infinitely many $CN$, but that 
  $$
\limsup_{n\,\,{\rm composite}}\frac{F(n)}{n}=1.
 $$
Notice that  by \cite[the estimate (2.8)]{ep}, we have  
that $F(n)/n^{1-\varepsilon}$ is unbounded on the composites for any
$\varepsilon >0$. On the other hand, by \cite[Theorem 6.1]{ep},
   $$
\limsup_{n\,\,{\rm composite}}\frac{F(n)\log^2 n}{n}>0.
 $$
A computation  (cf. Remark 2.77) suggests the following conjecture.
  \end{remark}
  \begin{conjecture} 
Let ${\mathcal WCN}_3$ denote the set of all weak Carmichael numbers
described in Remark $2.77$. Then 
    $$
\limsup_{n\in {\mathcal WCN}_3}\frac{F(n)}{n}=\frac{1}{3}.
 $$
 \end{conjecture}

 \section{Proofs of Propositions 2.3, 2.60, 2.63 and Corollary 2.20}

 \begin{proof}[Proof of Proposition $2.3$] 
Clearly, $\varphi(n)$ is  even for each $n\ge 3$, and the set $R_n$ 
can be presented as  
  \begin{equation}
R_n=\{r_1,r_2,\ldots,r_{\varphi(n)/2},n-r_1,
n-r_2,\ldots,n-r_{\varphi(n)/2}\}.
 \end{equation}
Using (3.1)   we find that
  \begin{equation}
 \sum_{\gcd(k,n)=1\atop 1\le k\le n-1}k^{n-1}
\equiv \sum_{i=1}^{\varphi(n)}r_i^{n-1}
=\sum_{i=1}^{\varphi(n)/2}(r_i^{n-1}+(n-r_i)^{n-1}).
   \end{equation}
 If $n$ is odd, the right hand side of (3.2) is
  $$
\equiv\sum_{i=1}^{\varphi(n)/2}(r_i^{n-1}+(-r_i)^{n-1})\pmod{n}=
0\pmod{n}.
 $$ 
Hence, every weak Carmichael number must be odd. Finally, if $n$ is odd,
using  (3.2), we have
 \begin{equation}
 \sum_{\gcd(k,n)=1\atop 1\le k\le n-1}k^{n-1}
= \sum_{i=1}^{\varphi(n)}r_i^{n-1}
=\sum_{i=1}^{\varphi(n)/2}(r_i^{n-1}+(n-r_i)^{n-1})
 \equiv 2\sum_{i=1}^{\varphi(n)/2}r_i^{n-1}\pmod{n}.
   \end{equation}
This together with Definition 2.1 concludes the proof.  
  \end{proof}
 \begin{proof}[Proof of Proposition $2.60$]
Applying the binomial formula and using the assumption that $n$ is an odd
composite integer, we find that
  \begin{equation*}\begin{split}
\sum_{\gcd(k,n)\atop 1\le k\le n-1}k^{n-1}
&=\sum_{i=1}^{\varphi(n)}r_i^{n-1}
=\sum_{i=1}^{\varphi(n)/2}(r_i^{n-1}+(n-r_i)^{n-1})\\
& \equiv \sum_{i=1}^{\varphi(n)/2}r_i^{n-1}-
{n-1\choose 1}n\sum_{i=1}^{\varphi(n)/2}r_i^{n-2}
+\sum_{i=1}^{\varphi(n)/2}r_i^{n-1}\pmod{n^2}\\
&=2\sum_{i=1}^{\varphi(n)/2}r_i^{n-1}+n\sum_{i=1}^{\varphi(n)/2}r_i^{n-2}
\equiv \varphi(n)\pmod{n^2}.
 \end{split}
\end{equation*}
This completes the proof.
  \end{proof}
 \begin{proof}[Proof of Corollary $2.20$]
Given product $n=pq$ with primes $q<p$, using Fermat little theorem, 
the sum  on the left right hand side of (2.8) in 
Definition 2.2 is 
      \begin{equation}\begin{split}
S(p,q):&=\sum_{\gcd(k,pq)=1\atop 1\le k\le pq-1}k^{pq-1}=
\sum_{j=0}^{q-1}\sum_{i=jp+1}^{(j+1)p-1}i^{pq-1}-
\sum_{s\equiv 0(\bmod{\,q})\atop 1\le s\le p-1}s^{pq-1}\\
&=\sum_{j=0}^{q-1}\sum_{i=jp+1}^{(j+1)p-1}i^{pq-1}-
q^{pq-1}\sum_{i=1}^{p-1}i^{pq-1}\\
& \equiv q\sum_{i=1}^{p-1}i^{pq-1}-q^{pq-1}\sum_{i=1}^{p-1}i^{pq-1}
\pmod{p}\\
&=q\sum_{i=1}^{p-1}i^{(p-1)q+(q-1)}-q^{(p-1)q+(q-1)}\sum_{i=1}^{p-1}i^{(p-1)q+(q-1)}\\
&\equiv q\sum_{i=1}^{p-1}i^{q-1}-q^{q-1}\sum_{i=1}^{p-1}i^{q-1}\pmod{p}\\
&=q(1-q^{q-2})\sum_{i=1}^{p-1}i^{q-1}\pmod{p}.
         \end{split}\end{equation}
Let $a$ be a generator of the multiplicative unit group of 
$\Bbb Z/p\Bbb Z$. Then since $q-1<p-1$ it follows easily that the set 
$\{1^{q-1},2^{q-1},\ldots, (p-1)^{q-1}\}$ regarding modulo $p$ 
conicides with the set $\Bbb Z_p^{*}=\{1,2,\ldots,p-1\}$ 
of all nonzero residues modulo $p$. This shows that 
     \begin{equation}
\sum_{i=1}^{p-1}i^{q-1}\equiv \sum_{i=1}^{p-1}i\pmod{p}
=\frac{(p-1)p}{2}\equiv 0\pmod{p}.
       \end{equation}
Taking (2.16) into (2.15), we immediately get 
    \begin{equation}
S(p,q)\equiv 0\pmod{p}.
   \end{equation}
      \begin{equation}
S(p,q):=\sum_{\gcd(k,pq)=1\atop 1\le k\le pq-1}k^{pq-1}\equiv 0\pmod{p}.
        \end{equation}
Since $\varphi(n)=(p-1)(q-1)\equiv 1-q(\bmod{\,p})$ from (2.18) 
we have 
  \begin{equation}
\sum_{\gcd(k,pq)=1\atop 1\le k\le pq-1}k^{pq-1}-\varphi(pq)\equiv q-1\pmod{p},
      \end{equation}
In view of the fact that $q<p$, we conclude that the expression on 
the left hand side of   (2.19) is not divisible by $p$. Theefore,
$n=pq$ is not a Carmichael number, and the proof is completed.
 \end{proof}

\begin{proof}[Proof of Proposition $2.63$] 
Using Euler totient theorem, we have
   \begin{equation}    
\sum_{1\le k\le p^f-1\atop (k,p)=1}k^{p^f-1}\equiv
\sum_{1\le k\le p^f-1\atop (k,p)=1}
\frac{1}{k^{p^{2f}-p^{2f-1}-p^f+1}}\pmod{p^{2f}}.
  \end{equation}
By the  congruence  (6) in \cite{sl1}, 
it follows that if $s$ in an even positive integer 
and $t=(\varphi(p^{2f})-1)s$ then 
   \begin{equation}    
\sum_{1\le k\le p^f-1\atop (k,p)=1}\frac{1}{k^{p^{2f}-p^{2f-1}-p^f+1}}\equiv
p^fB_t\pmod{p^{2f}}
 \end{equation}
where 
   \begin{equation}    
t=(\varphi(p^{2f})-1)(p^{2f}-p^{2f-1}-p^f+1)=
(p^{2f}-p^{2f-1}-1)(p^{2f}-p^{2f-1}-p^f+1)
   \end{equation}
Comparing (3.10) and (3.11) gives
 \begin{equation}    
\sum_{1\le k\le p^f-1\atop (k,p)=1}k^{p^f-1}\equiv p^fB_t\pmod{p^{2f}}
    \end{equation}
 with $t$ given by (3.12).
Notice that by von Staudt-Clausen's theorem,
the denominator $D_t$  of  Bernoulli number $B_t=N_t/D_t$ 
is the product of all primes $p$ such that $p-1$ divides $t$.
In particular, this shows that $p\parallel D_t$. 
From this and the congruence (3.13) we conclude that 
$\sum_{1\le k\le p^f-1\atop (k,p)=1}k^{p^f-1}$ is divisible by
$p^{2f}$ if and only if $N_t$ is divisible by $p^{f+1}$.
  \end{proof}

\section{Proof of Theorem 2.4 and Corollary 2.17}
Proof of Theorem 2.4 is based on the following three auxiliary results.
 \begin{lemma} 
Let $p^e$ be a  power of an odd prime $p$,  and let $me$ be a positive integer 
such that $m$ is not divisible by $p-1$. Then
   \begin{equation}
\sum_{1\le k\le p^e-1\atop \gcd(k,p)=1}k^m\equiv 0\pmod{p^e}.
   \end{equation}
 \end{lemma}

\begin{proof}
Let $a$ be a 
primitive root modulo $p^e$.
Then $a^m$ is not divisible by $p$. Moreover, it is easy to see 
that the set $\{aj:\,1\le j\le p^e-1\quad {\rm and}\quad 
\gcd(j,p)=1\}$ reduced modulo $p^e$ coincides 
with the set of all residues modulo $p^e$ which are relatively 
prime to $p$.  This shows that 
   \begin{equation*}
a^m\sum_{1\le k\le p^s-1\atop \gcd(k,p)=1}k^m=
\sum_{1\le k\le p^e-1\atop \gcd(k,p)=1}(ak)^m
\equiv  \sum_{1\le k\le p^e-1\atop \gcd(k,p)=1}k^m\pmod{p^e},
   \end{equation*}
whence it follows that 
   $$
(a^m-1)\sum_{1\le k\le p^e-1\atop \gcd(k,p)=1}k^m\equiv 0\pmod{p^e}.
  $$
This together with the assumption that $a^m\not\equiv 1(\bmod{\,p^e})$
immediately yields the desired congruence (4.1).
 \end{proof}

\begin{lemma} Let $n>1$ be a positive integer 
with a prime factorization  $n=p_1^{e_1}p_2^{e_2}\cdots p_s^{e_s}$ 
 where $s\ge 2$. Let $R$ be a set of positive integers less than $n$ and 
relatively prime to $n$. For any fixed $i\in \{1,2,\ldots,s\}$ set 
 $$
R(p_i^{e_i})=\{a\in \Bbb N: 1\le a\le p_i^{e_i}\quad  and \quad \gcd(a,p_i)=1\}.
  $$ 
  For all pairs  $(i,j)$ with  $i\in \{1,2,\ldots,s\}$ and  
$j\in R(p_i^{e_i})$ define the set $A_{ij}$ as
  $$
A_{ij}=\{a\in R:\, a\equiv j(\bmod{\,p_i^{e_i}})\}.  
  $$
Then for any $i\in \{1,2,\ldots,s\}$
    \begin{equation}
|A_{ij}|=\varphi\left(\frac{n}{p_i^{e_i}}\right)=
\prod_{1\le l\le s\atop l\not= i}(p_l^{e_l}-p_l^{e_l-1})
\quad  for\,\, all\quad
j\in R(p_i^{e_i}),
  \end{equation}
where $|S|$ denotes the cardinality of a finite set $S$.
 \end{lemma}
\begin{proof} 
Clearly, it  suffices to show that (4.2)  is satisfied for  
$i=1$. Then for any fixed 
$j\in R(p_1^{e_1})$ consider the set
  $$
T_j=\{j+rp_1^{e_1}:\, r=0,1,\ldots ,\frac{n}{p_1^{e_1}}-1
(=p_2^{e_2}\cdots p_s^{e_s}-1)\}.
  $$ 
Since for each $t\in T_j$,  
  $
j\le t\le j+p_1^{e_1}\left(\frac{n}{p_1^{e_1}}-1\right)\le n-1,
  $
it follows that the set  $A_{1j}$   actually consists of these  
elements in $T_j$ that are relatively prime to $n/p_1^{e_1}$. Notice that
the set $T_j$ reduced modulo $n/p_1^{e_1}$ coincides with the set 
$\{0,1,2,\ldots,n/p_1^{e_1}-1\}$ of all residues modulo $n/p_1^{e_1}$;
namely, if  $j+r_1p_1^{e_1}\equiv j+r_2p_1^{e_1}(\bmod{\,n/p_1^{e_1}})$ with 
$0\le r_1<r_2\le n/p_1^{e_1}-1$, then $n/p_1^{e_1}\mid (r_2-r_1)p_1^{e_1}$, 
and so, $n/p_1^{e_1}\mid (r_2-r_1)$ which is impossible because of 
$1\le r_2-r_1\le n/p_1^{e_1}-1$. This shows that
the set $A_{1j}$ contains exactly $\varphi(n/p_1^{e_1})$ elements, which is 
equal to $(p_2^{e_2}-p_2^{e_2-1})\cdots (p_s^{e_s}-p_s^{e_s-1})$. This 
completes the proof.
 \end{proof}
 \begin{lemma}
Let $n$ and $e$ be positive integers and let $p$ be a prime such that
  $p^e\mid n$ and $p-1\mid n-1$. Then 
   \begin{equation}
k^{n-1}\equiv k^{p^{e-1}-1}\pmod{p^e}
  \end{equation}
for every integer $k$ that is not divisible by $p$.  
\end{lemma}
   \begin{proof}
Take $n=p^en'$ with an integer $n'$.
Then from the assumption    $p-1\mid n-1$ it follows that 
$n'\equiv 1 (\bmod{\,(p-1)})$, and therefore
 $$
n-1=p^en'-1\equiv p^{e-1}-1\pmod{p^{e-1}(p-1)}.
  $$
If $k$ is not divisible by $p$, then from the above congruence, the fact that 
$\varphi(p^e)=p^{e-1}(p-1)$ and Euler  totient theorem, we have 
$k^{p^{e-1}-1}\equiv k^{n-1}(\bmod{\, p^{e}})$ which immediately implies  
(4.3).
   \end{proof}   
  \begin{lemma} Let $n>1$ be a positive integer 
with a prime factorization  $n=p_1^{e_1}p_2^{e_2}\cdots p_s^{e_s}$. 
For any fixed $i\in \{1,2,\ldots,s\}$, set 
 $$
R(p_i^{e_i})=\{a\in \Bbb N: 1\le a\le p_i^{e_i}\quad and \quad \gcd(a,p_i)=1\}.
  $$ 
Then for every $i\in\{1,2,\ldots,s\}$ there holds
    \begin{equation}
\sum_{\gcd(k,n)=1\atop 1\le k\le n-1}k^{n-1}
\equiv \varphi\left(\frac{n}{p_i^{e_i}}\right)\sum_{k\in R(p_i^{e_i})}k^{n-1}
\pmod{p_i^{e_i}}.
  \end{equation}
If in addition, $p_i-1\mid n-1$ for some $i\in \{1,2,\ldots,s\}$, then
      \begin{equation}
\sum_{\gcd(k,n)=1\atop 1\le k\le n-1}k^{n-1}
\equiv \varphi\left(\frac{n}{p_i^{e_i}}\right)\sum_{k\in R(p_i^{e_i})}
k^{p_i^{e_i-1}-1}
\pmod{p_i^{e_i}}.
  \end{equation}
\end{lemma}

  \begin{proof} 
Consider the set $R$ of all reduced residues modulo $n$ that are relatively prime to $n$, 
i.e.,
 $$
R=\{k:\,1\le k\le n-1\,\, {\rm and}\,\,\gcd(k,n)=1\}.
  $$
Let $i\in \{1,2,\ldots,m\}$ be any fixed. 
For each $j\in R(p_i^{e_i})$   take 
  $$
A_{ij}=\{a\in R:\, a\equiv j(\bmod{\,p_i^{e_i}})\}.
  $$
Then by Lemma 4.2, 
  \begin{equation}
|A_{ij}|=\varphi\left(\frac{n}{p_i^{e_i}}\right)\quad {\rm for\,\, all}\quad
j\in R(p_i^{e_i}).
  \end{equation}
Furthermore, using (4.3) and (4.1) of Lemma 4.1 with $s=n-1$, we have
   \begin{equation*}\begin{split}
 \sum_{\gcd(k,n)=1\atop 1\le k\le n-1}k^{n-1}
&=\sum_{k\in R}k^{n-1}=\sum_{j\in R(p_i^{e_i})}\sum_{k\in A_{ij}}k^{n-1}\\
&\equiv \varphi\left(\frac{n}{p_i^{e_i}}\right)\sum_{k\in R(p_i^{e_i})}k^{n-1}
\pmod{p_i^{e_i}}.
     \end{split}\end{equation*}
The above congruence implies (4.4). Finally, if $p_i-1\mid n-1$ then 
substituting (4.3) of Lemma 4.3 with $p_i^{e_i}$ instead of $p^e$ 
into (4.4) we immediately obtain (4.5).
    \end{proof}
 \begin{lemma} Let $n$ be a composite positive integer with the prime 
factorization $n=p_1^{e_1}\cdots p_s^{e_s}$. 
For any fixed $i\in \{1,2,\ldots,s\}$ take 
 $$
R(p_i^{e_i})=\{a\in \Bbb N: 1\le a\le p_i^{e_i}\quad  and\quad 
\gcd(a, p_i)=1\}.
  $$ 
Then $n$ is a weak Carmichael number if and only if for every 
$i\in\{1,2,\ldots ,s\}$ there holds
   \begin{equation}
\left(\sum_{k\in R(p_i^{e_i})}k^{n-1}+p_i^{e_i-1}\right)\varphi
\left(\frac{n}{p_i^{e_i}}\right)\equiv 0\pmod{p_i^{e_i}}.
  \end{equation}
 \end{lemma}
  \begin{proof}
Clearly, $n=p_1^{e_1}\cdots p_s^{e_s}$ is a weak Carmichael number 
 if and only if (2.2)  is satisfied  modulo $p_i^{e_i}$ for all 
$i=1,\ldots,s$ which is by (4.3) of Lemma 4.3 (4.4) of Lemma 4.4
equivalent to the congruence 
   \begin{equation}
\varphi\left(\frac{n}{p_i^{e_i}}\right)\sum_{k\in R(p_i^{e_i})}k^{n-1}\equiv 
\varphi(n)\pmod{p_i^{e_i}}.
  \end{equation}
Since 
 $
\varphi(n)=(p_i^{e_i}-p_i^{e_i-1})\varphi(n/p_i^{e_i})\equiv  
-p_i^{e_i-1}\varphi(n/p_i^{e_i})(\bmod{\,p_i^{e_i}}),
 $
substituting this into (4.8), we immediately obtain (4.7).
  \end{proof}
 \begin{lemma} 
Let $p\ge 3$ be a a prime and let $e\ge 2$ be a positive integer. Define
 $$
R(p^e)=\{a\in \Bbb N:\, 1\le a\le p^e\quad  and\quad 
\gcd(a, p)=1\}.
  $$ 
Then 
   \begin{equation}
\sum_{k\in R(p^e)}k^{p^{e-1}-1}\equiv\sum_{k=1}^{p^e-1}k^{p^{e-1}-1}\pmod{p^e}.
   \end{equation}
  \end{lemma}
 \begin{proof}
From the obvious inequality $p^{e-1}-1\ge e$ with $p\ge 3$ and $e\ge 2$
we see that every term in the sum on the right hand side of (4.9)
that  is divisible by $p$ is also divisible by $p^e$. This yields the desired 
congruence (4.9). 
  \end{proof}

 The following congruence is known as  
a {\it  Carlitz-von Staudt's result} \cite{ca} in 1961 
(for an easier proof see \cite[Theorem 3]{mo3}).
 \begin{lemma} 
Let $p\ge 3$ be a a prime and let $e$ and $l$ be  positive integers
such that $p-1$ does not divide $l$. Then 
   \begin{equation}
\sum_{k\in R(p^e)}k^l \equiv 0\pmod{p^e}.
  \end{equation}
 \end{lemma}
  \begin{proof}
By a particular case of a result obtained in 1955 by H.J.A. Duparc and 
W. Peremans \cite[Theorem 1]{dp} (cf. \cite[Corollary 2]{sl1} or 
\cite[the congruence (58) in Section 8]{me3}), 
if $r$ is a positive integer such that $p-1$ does not divide $r$,
then
    \begin{equation}
\sum_{k\in R(p^e)}\frac{1}{k^r}\equiv 0\pmod{p^e}.
  \end{equation}
Letting $r=\varphi(p^e)-l=p^e-p^{e-1}-l$ in (4.11) and then applying 
Euler totient theorem modulo $p^e$, we immediately obtain (4.10).
    \end{proof}
  \begin{lemma}{\rm(\cite{ca}, \cite[Theorem 3]{mo3})} 
Let $l$ and $m\ge 2$ be positive integers. Then
   \begin{equation}
S_l(m):=\sum_{i=1}^{m-1}i^l\equiv\left\{
  \begin{array}{ll}
0  \pmod{\frac{(m-1)m}{2}}&  if\,\, l\,\, is\,\, odd\\
-\sum_{(p-1)\mid l, p\mid m}\frac{m}{p}  \pmod{m} & 
if\,\, l\,\, is\,\, even
\end{array}\right.\end{equation}
where the summation is taken over all primes $p$ such that 
$(p-1)\mid l$ and $p\mid m$.
  \end{lemma}

  \begin{proof}[Proof of Theorem $2.4$] 
First assume that  $n=p_1^{e_1}p_2^{e_2}\cdots p_s^{e_s}$ is a 
composite  integer such that $p_i-1\mid n-1$ for all 
$i=1,2,\ldots, s$. Then by Lemma 4.5 $n$ is a weak Carmichael number if  
 for every $i\in\{1,2,\ldots ,s\}$ there holds
   \begin{equation}
\sum_{k\in R(p_i^{e_i})}k^{n-1}+p_i^{e_i-1}\equiv 0\pmod{p_i^{e_i}}.
  \end{equation}
For any fixed $i\in\{1,2,\ldots, s\}$
consider two cases: $e_i=1$ and $e_i\ge 2$.

{\it Case} 1. $e_i=1$. Then since $n-1\equiv 0 (\bmod{\,(p-1)})$,
using Fermat little theorem, we obtain
   \begin{equation}
\sum_{k\in R(p_i)}k^{n-1}+p_i^{1-1}=
\sum_{k=1}^{p_i-1}k^{n-1}+1\equiv (p_i-1)+1\equiv 0\pmod{p_i}.
  \end{equation}
Therefore, (4.13) is satisfied. 

{\it Case} 2. $e_i\ge 2$.
Then using (4.3) of Lemma 4.3 and (4.10) of Lemma 4.7,
with the notations of Lemma 4.4, for any fixed $i\in\{1,2,\ldots, s\}$ we have 
   \begin{equation}\begin{split}
\sum_{k\in R(p_i^{e_i})}k^{n-1}&\equiv \sum_{k\in R(p_i^{e_i})}
k^{p_i^{e_i-1}-1}\pmod{p_i^{e_i}}\\
&\equiv \sum_{k=1}^{p_i^{e_i}-1}k^{p_i^{e_i-1}-1}\pmod{p_i^{e_i}}.
  \end{split}\end{equation}
Furthermore, by the second congruence of (4.12) of 
Lemma 4.8 (cf. \cite[Lemma 1]{gm3}), we immediately get 
   \begin{equation}\begin{split}
 \sum_{k=1}^{p_i^{e_i}-1}k^{p_i^{e_i-1}-1}\equiv -\frac{p_i^{e_i}}{p_i}
=-p_i^{e_i-1}\pmod{p_i^{e_i}}.
  \end{split}\end{equation}
Comparing (4.15) and (4.16) immediately gives (4.13), and therefore,
$n$ is a weak Carmichael number.

Conversely, suppose that $n$ is a weak Carmichael number  with 
a prime factorization $n=p_1^{e_1}\cdots p_s^{e_s}$ 
where $p_1,\ldots,p_s$ are odd primes. Suppose that 
for some  $i\in\{1,2,\ldots,s\}$,  $n-1$ is not divisible 
by $p_i-1$. Then by (4.10) of Lemma 4.7, we have
     \begin{equation}
\sum_{k\in R(p_i^{e_i})}k^{n-1}\equiv 0\pmod{p^e}.
  \end{equation}
Substituting (4.17) in (4.7) of Lemma 4.5,   we get
  \begin{equation}
p_i^{e_i-1}\varphi
\left(\frac{n}{p_i^{e_i}}\right)\equiv 0\pmod{p_i^{e_i}},
  \end{equation}
or equivalently,
   \begin{equation}
\varphi\left(\frac{n}{p_i^{e_i}}\right)\equiv 0\pmod{p_i}.
  \end{equation}
The above congruence implies that 
$p_i\mid \prod_{1\le j\le s\atop j\not= i}p_j^{e_j-1}(p_j-1)$.
It follows that $p_i\mid p_j-1$ for some $j\not= i$. 

  We can choose such a $p_i$ to be maximal, i.e.,
     $$
p_i=\max_{1\le t\le s}\{p_t:\, n-1\not\equiv 0\pmod{p_t-1}\}.
    $$
Then, as it is proved previously, we must have 
$p_i\mid p_j-1$ for some $j\not= i$. It follows that $p_i<p_j$, 
and hence, by the maximality of $p_i$ we conclude that 
$p_j-1\mid n-1$. Therefore, $p_i\mid n-1$, which is impossible 
because of $p_i\mid n$. A contradiction, and hence $n-1$ is
divisible by $p_i-1$ for each $i=1,2,\ldots,s$. This completes the proof
of Theorem 2.4.
  \end{proof}

  \begin{proof}[Proof of Corollary $2.17$] 
If $n$ is a weak Carmichael number that is not Carmichael number, 
then $\mu(n)=0$, and hence the congrueence in Corollary 2.17 reduced to the 
congruence (2.2).

Conversely, if $n>1$ satisfies the congruence 
  \begin{equation}
\sum_{\gcd(k,n)=1\atop 1\le k\le n-1}k^{n-1}\equiv \varphi(n)+\mu(n) \pmod{n},
  \end{equation}
then consider two cases:
1) $n$ is not  square-free and 2) $n$ is  square-free.
In the first case (4.20) becomes
   \begin{equation}
\sum_{\gcd(k,n)=1\atop 1\le k\le n-1}k^{n-1}\equiv \varphi(n)\pmod{n},
  \end{equation}
whence using  Theorem 2.4 we conclude that $n$ is a weak 
Carmichael number. 

In the second case, we have $\mu(n)=\pm 1$, and then (4.20) becomes
     \begin{equation}
\sum_{\gcd(k,n)=1\atop 1\le k\le n-1}k^{n-1}\equiv \varphi(n)\pm 1\pmod{n}.
  \end{equation}

In view of Corollary 2.16, the above congruence shows that $n$ is not 
a weak Carmichael number.
Then by Theorem 2.4 there exists 
a prime factor $p$ of $n$ such that $p-1$ does not divide $n-1$.
Asuume that $n=p^en'$ with $n'$ such that $p$ does not divide $n$. 
Then applying (4.4) and Fermat little theorem,  we find that 
     \begin{equation}\begin{split}
\sum_{\gcd(k,n)=1\atop 1\le k\le n-1}k^{n-1}
&\equiv \varphi\left(\frac{n}{p^e}\right)\sum_{k\in R(p^{e})}k^{n-1}\pmod{p}\\
&\equiv \varphi\left(\frac{n}{p^e}\right)|R(p^{e})|
=\varphi\left(\frac{n}{p^e}\right)\varphi(p^e)=\varphi(n)\pmod{p}.
       \end{split}\end{equation}
Substituting the congruence (4.23) into (4.22) reduced modulo $p$,
we obtain $0\equiv \pm 1 (\bmod{\,p})$. A contradiction, and hence, 
$n$ is a weak Carmichael number which is not a Carmichael number. 
This concludes the proof.
 \end{proof}

     \end{document}